\documentclass[3p]{elsarticle}
\usepackage{amsfonts,amssymb,amsmath,amsthm}	
\usepackage[english]{babel}                 	
\usepackage[T1]{fontenc}                    	
\usepackage{ae}                             	
\usepackage{setspace}      										
\usepackage{longtable}     										
\usepackage{graphicx}      										
\usepackage{tikz}          										
\usepackage[permil]{overpic}       						
\usepackage{fancybox}													
\usepackage{ifthen}														
\usepackage{subcaption}										 	
\usepackage{tikz}
\usepackage{tikz-3dplot}
\usepackage[numbers]{natbib}
\usepackage{fancyhdr, graphicx}
\usepackage{pgfgantt}
\usepackage{bm}

\newtheorem{theorem}{Theorem}[section]
\newtheorem{lemma}[theorem]{Lemma}

\newtheorem{remark}{Remark}

\usepackage{amssymb}

\journal{arXiv.org}

\newcommand{\TheTitle}{A low-rank power iteration scheme for neutron transport criticality problems} 

\date{\today}

\usepackage{amsopn}

\newcommand{\keff}{{k_\mathrm{eff}}}
\journal{arXiv}

\begin{document}
\begin{frontmatter}

\title{\TheTitle}

\author[adressJonas]{Jonas Kusch}
\author[adressRyan]{Benjamin Whewell}
\author[adressRyan]{Ryan McClarren}
\author[adressJonas]{Martin Frank}

\address[adressJonas]{Karlsruhe Institute of Technology, Karlsruhe, Germany, \{jonas.kusch, martin.frank\}@kit.edu}
\address[adressRyan]{University of Notre Dame, Notre Dame, Indiana, USA, \{bwhewell, rmcclarr\}@nd.edu}

\begin{abstract}
Computing effective eigenvalues for neutron transport often requires a fine numerical resolution. The main challenge of such computations is the high memory effort of classical solvers, which limits the accuracy of chosen discretizations. In this work, we derive a method for the computation of effective eigenvalues when the underlying solution has a low-rank structure. This is accomplished by utilizing dynamical low-rank approximation (DLRA), which is an efficient strategy to derive time evolution equations for low-rank solution representations. The main idea is to interpret the iterates of the classical inverse power iteration as pseudo-time steps and apply the DLRA concepts in this framework. In our numerical experiment, we demonstrate that our method significantly reduces memory requirements while achieving the desired accuracy. Analytic investigations show that the proposed iteration scheme inherits the convergence speed of the inverse power iteration, at least for a simplified setting.
\end{abstract}


\end{frontmatter}

\section*{Introduction}
In analyzing nuclear systems the eigenvalue problem that describes the behavior of a neutron-induced fission chain reaction known as the $k$-eigenvalue problem is of fundamental importance.  The magnitude of the dominant eigenvalue indicates whether the fission chain reaction will 1) ultimately diverge, 2) reach a constant, non-zero steady state, or 3) decay to zero. These cases are known as supercritical, critical, or subcritical systems, respectively. 

Given that the dominant (i.e., maximal) eigenvalue determines the character of the system, the inverse power iteration method (and its variations) is the most commonly applied method. In this method the neutron transport operator is repeatedly applied to an initial guess of the eigenvector in a similar manner to solving a steady-state problem where the righthand side changes with each iteration \cite{mcclarren2017computational}. 

There are a variety of mathematical models for the transport of neutrons including Monte Carlo \cite{lux2018monte}, discrete ordinates (S$_N$) \cite{lewis1984computational}, spherical harmonics (P$_N$) \cite{bell1970nuclear}, and simplified  P$_N$ methods \cite{mcclarren2010theoretical}. The workhorse model for nuclear systems is the multigroup diffusion method \cite{stacey2018nuclear} where the neutron energies are discretized into finite energy ranges known as groups, and an elliptic, diffusion operator is used to approximate the migration of neutrons. Such a model will require storage of a solution that has a size that is the product of the energy and spatial degrees of freedom.  When dealing with a highly heterogeneous system such as a nuclear reactor, the number of spatial degrees of freedom can be enormous \cite{wang2009three}, and to correctly approximate the reaction rates the number energy degrees of freedom can easily reach several hundred \cite{hazama2006development}. For this reason, many problems are approximated by coarsened descriptions in space and energy to enable exploration of design space in engineering applications. These coarse problem descriptions can be accurate for systems that have either historical data to calibrate to or where past high-fidelity calculations can be used to inform discretizations.  Nevertheless, the promise of small modular nuclear systems, interplanetary exploration scale reactors, and other emerging technologies may not be able to rely on these techniques.


To enable high-fidelity calculations we seek numerical techniques that have reduced  memory and computational costs. An efficient method with low memory requirements for time-dependent problems is dynamical low-rank approximation (DLRA). This method has been introduced for matrix-valued solutions in \cite{KochLubich07} and an extension to tensors is given in \cite{KochLubich10}. Its main idea is to represent and evolve the solution on a manifold of rank $r$ functions. DLRA yields evolution equations for the individual factors of the solution \cite{KochLubich07}, which can be solved numerically with standard methods. Robust integrators for these equations are the matrix projector-splitting integrator \cite{LubichOseledets} and the unconventional integrator \cite{CeL21}. The original use of dynamical low-rank approximation focuses on matrix ordinary differential equations. Note that partial differential equations can be brought into such a form by performing a discretization of the phase-space except for time. In \cite{EiL18}, the derivation of DLRA has been extended to function spaces, i.e., the evolution equations can be derived before discretizing the problem. This approach does not only provide the ability to derive stable discretizations \cite{kusch2021stability}, but in the radiation transfer context allows an efficient implementation of scattering \cite{kusch2021stability,kusch2021robust}. Problems in which dynamical low-rank is successfully applied to reduce memory and computational costs are, e.g., kinetic theory \cite{EiL18,einkemmer2019quasi,PeMF20,PeM20,einkemmer2020low,EiHY21,EiJ21,guo2021low,kusch2021robust} as well as uncertainty quantification \cite{FeL18,MuN18,MuNV20,SaL09,kusch2021DLRUQ}. Furthermore, DLRA allows for adaptive model refinement \cite{DeRV20,ceruti2021rank,rodgers2020adaptive}, where the main idea is to pick the rank of the solution representation adaptively. Recently, an approach to employ DLRA for computing rightmost eigenpairs has been proposed in \cite{guglielmi2021computing}.

In this work, we combine dynamical low-rank approximation and the neutron diffusion equations to obtain an inverse power iteration scheme with low memory requirements. For this, we treat the updates of the power iteration as pseudo-timesteps. DLRA is then applied to the iteration scheme. I.e., the solution is represented as a product of rank $r$ matrices and the resulting DLRA update equations yield an iteration scheme of each matrix. The memory requirement of the solution representation reduces from $N_x\times G$, where $N_x$ is the number of spatial cells and $G$ is the number of energy groups to $N_x\times r + G\times r$, where $r\ll N_x,G$.

This paper is structured as follows: After the introduction, we provide an overview over the used concepts in Section~\ref{sec:background}: Section~\ref{sec:backgroundKeig} presents a discretization of the full $k$-eigenvalue problem for matrix solutions $\bm\phi\in\mathbb{R}^{N_x\times G}$. Section~\ref{sec:DLRAIntro} shortly reviews important concepts of dynamical low-rank approximation used in this work. In Section~\ref{sec:DLRAPowerIt} we employ dynamical low-rank approximation to derive a memory efficient iteration scheme for the $k$-eigenvalue problem. Section \ref{sec:convergence} gives a convergence proof for the scheme in a simplified setting. Numerical results are presented in Section \ref{sec:numerical} before concluding with Section \ref{sec:conclusion}.

\section{Background}\label{sec:background}
In the following, we give a short overview of the concepts employed in this work. Furthermore, we use this section to write the energy-dependent neutron diffusion equation in a compact matrix notation, which simplifies the derivation of DLRA evolution equations.
\subsection{$k$-eigenvalue computation}\label{sec:backgroundKeig}
Our main goal is to determine the maximum eigenvalue $\keff$ as well as its corresponding eigenfunction $\phi$ of the neutron transport $k$-eigenvalue problem. The corresponding multigroup approximation reads
\begin{align}\label{eq:keigenmultigroup}
    - \nabla \cdot D_g(r)\nabla\phi_g(r) + \Sigma_{t,g}(r)\phi_g(r) = \frac{\chi_g}{\keff}\sum_{g'}\nu\Sigma_{f,g'}(r)\phi_{g'}(r) + \sum_{g'}\Sigma_{s,g',g}(r) \phi_g(r),
\end{align}
where $r\in \Omega\subset \mathbb{R}^{d}$ is the spatial domain, $\Sigma_{t,g}(r)$ is the total cross section of energy group $g$ at spatial position $r$, $\Sigma_{f,g}(r)$ is the fission cross-section and $\Sigma_{s,g',g}(r)$ is the scattering cross section between groups $g$ and $g'$. Furthermore, $\nu$ is the mean number of particles produced per fission event and $\chi_g$ is the fission neutron distribution function for group $g$. We are interested in computing $\phi_g(r)$ which is the integral of the scalar flux over the energy range of group $g$ at position $r$ corresponding to the maximal eigenvalue $\keff$. For sake of readability, we omit boundary conditions and state them whenever needed.

In the following, our aim is to state a numerical discretization of \eqref{eq:keigenmultigroup}, which treats the scalar flux at given spatial points as a matrix. That is, for a spatial grid $r_1,\cdots,r_{N_x}$ and energy groups $g\in\{1,\cdots,G\}$ we have $\bm{\phi} = (\phi_{jg})_{j,g  = 1}^{N_x,G}$, where $\phi_{jg} = \phi_g(x_j)$. Such a formulation enables the use of dynamical low-rank approximation to reduce memory requirements. First, the different material coefficients can be written as
\begin{align}\label{eq:expansionD}
    D_g(r) = \sum_{\ell = 1}^{N_m} \rho_{\ell}(r)D_{g}^{(\ell)},
\end{align}
where $N_m$ is the number of materials and $D_{g}^{(\ell)}$ is the diffusion coefficient (or any other material coefficient) for material $\ell$ and energy group $g$. The functions $\rho_{\ell}(r)$ denote densities of material $\ell$. When evaluating $D_g$ at cell interfaces $r_{j+1/2} = (r_j+r_{j+1})/2$, one needs to approximate the the material coefficient though its harmonic mean, i.e.,
\begin{align*}
    D_g(r_{j+1/2}) =: D_{g,j+1/2} = 2\frac{D_{g,j} D_{g,j+1}}{D_{g,j} + D_{g,j+1}} = 2\frac{\sum_{\ell,k}\rho_j^{(\ell)}D_g^{(\ell)}\rho_{j+1}^{(k)}D_g^{(k)}}{\sum_{\ell}(\rho_j^{(\ell)}+\rho_{j+1}^{(\ell)})D_g^{(\ell)}}.
\end{align*}
Noting that only two terms in the sums can be non-zero, we get
\begin{align}\label{eq:approxD}
    D_{g,j+1/2} = \sum_{\ell,k=1}^{N_m}\rho_j^{(\ell)}\rho_{j+1}^{(k)}\left(\rho_j^{(\ell)}+\rho_{j+1}^{(k)}\right)\frac{D_g^{(\ell)}D_g^{(k)}}{D_g^{(\ell)}+D_g^{(k)}}.
\end{align}
The diffusion operator in one dimension is discretized through
\begin{align*}
    \nabla \cdot D_g(r)\nabla\phi_g(r)\Big|_{r=r_j} \approx \left(\bm D(g)\bm \phi_g \right)_j
\end{align*}
where $\bm \phi_g\in\mathbb{R}^{N_x}$ collects the scalar flux at all spatial cells. The matrix $\bm D(g)\in\mathbb{R}^{N_x\times N_x}$ has values
\begin{align*}
    D_{j,j\pm 1}(g) &= \pm\frac{1}{\Delta r\cdot V_j} D_{g,j\pm 1/2}S_{j\pm 1/2}, \\
    D_{j,j}(g) &= -\frac{1}{\Delta r\cdot V_j} \left[D_{g,j+1/2}S_{j+1/2}+D_{g,j-1/2}S_{j-1/2}\right],
\end{align*}
where $\Delta r$ is the size of each radial element. The surface area between cell $j$ and $j\pm 1$ is denoted by $S_{j\pm 1/2}$ and the area of cell $j$ is denoted by $V_j$. The choice of these terms defines the spatial geometry. In our numerical experiments, we look at spherical domains, where we have $V_j = \frac{4\pi}{3}(r_{i+1/2}^3-r_{i-1/2}^3)$ and $S_{j\pm 1/2}=4\pi r_{i\pm 1/2}^2$.
Using \eqref{eq:approxD} in the above definition of $\bm D(g)$, lets us write $\bm{D}(g)\bm{\phi}_g$ as
\begin{align*}
  \bm{D}(g)\bm{\phi}_g = \sum_{\ell,k = 1}^{N_m} \frac{D_g^{(\ell)}D_g^{(k)}}{D_g^{(\ell)}+D_g^{(k)}}\bm D^{(\ell,k)}\bm \phi_g,
\end{align*}
where we use
\begin{align*}
    D^{(\ell,k)}_{j,j\pm 1} &= \pm\frac{\rho_j^{(\ell)}\rho_{j\pm 1}^{(k)}}{\Delta r\cdot V_j}  (\rho_{\ell}(r_{j})+\rho_{k}(r_{j\pm 1}))S_{j\pm 1/2}, \\
    D^{(\ell,k)}_{j,j} &= -\frac{1}{\Delta r\cdot V_j} \left[ \rho_j^{(\ell)}\rho_{j+1}^{(k)}\left(\rho_j^{(\ell)}+\rho_{j+1}^{(k)}\right)S_{j+1/2}+ \rho_j^{(\ell)}\rho_{j-1}^{(k)}\left(\rho_j^{(\ell)}+\rho_{j-1}^{(k)}\right)S_{j-1/2}\right].
\end{align*}
Let us further use the diagonal matrix $\bm{\rho}^{(\ell)}\in\mathbb{R}^{N_x\times N_x}$ with entries $\rho^{(\ell)}_{j j} = \rho_{\ell}(r_j)$ to write
\begin{align*}
   \frac{\chi_g \nu}{k}\sum_{g'}\Sigma_{f,g'}(r)\phi_{g'}(r) &= \frac{\chi_g\nu}{k}\sum_{g',\ell} \rho_{\ell}(r_j)\Sigma_{f,g'}^{(\ell)}\phi_{g'}(r_j) = \frac{\chi_g\nu}{k}\sum_{g',\ell} \Sigma_{f,g'}^{(\ell)}\bm{\rho}^{(\ell)}\bm \phi_{g'}, \\
   \sum_{g'}\Sigma_{s,g',g}(r_j) \phi_g(r_j) &= \sum_{g',\ell}\rho_{\ell}(r_j)\Sigma_{s,g',g}^{(\ell)} \phi_g(r)= \sum_{g',\ell}\Sigma_{s,g',g}^{(\ell)}\bm{\rho}^{(\ell)}\bm \phi_{g'}.
\end{align*}
Then, for a given group $g$, the diffusion equation reads
\begin{align*}
  -\sum_{\ell,k = 1}^{N_m} \frac{D_g^{(\ell)}D_g^{(k)}}{D_g^{(\ell)}+D_g^{(k)}}\bm D^{(\ell,k)}\bm \phi_g + \sum_{\ell = 1}^{N_m}\Sigma_{t,g}^{(\ell)}\bm{\rho}^{(\ell)}\bm\phi_g = \frac{\chi_g\nu}{\keff}\sum_{g',\ell} \Sigma_{f,g'}^{(\ell)}\bm{\rho}^{(\ell)}\bm \phi_{g'}+\sum_{g',\ell}\Sigma_{s,g',g}^{(\ell)}\bm{\rho}^{(\ell)}\bm \phi_{g'}.
\end{align*}
We can write this term as a bigger system for $\bm\phi\in\mathbb{R}^{N_x\times G}$. For this, we define $\bm{\widetilde\Sigma}_{f}^{(\ell)}= \left(\chi_g\nu\Sigma_{f,g'}^{(\ell)}\right)_{g,g'=1}^G$ and the diagonal matrices $\bm M^{(\ell,k)}\in\mathbb{R}^{G\times G}$ with entries $M_{gg}^{(\ell,k)} = D_g^{(\ell,k)}$ as well as $\bm{\Sigma}_t^{(\ell)}\in\mathbb{R}^{G\times G}$ with $\Sigma_{t,gg}^{(\ell)}=\Sigma_{t,g}^{(\ell)}$. Then we have
\begin{align*}
    -\sum_{\ell,k} \bm D^{(\ell,k)}\bm \phi \bm M^{(\ell,k)} + \sum_{\ell}\bm{\rho}^{(\ell)}\bm\phi\bm{\Sigma}_t^{(\ell)} = \frac{1}{\keff}\sum_{\ell} \bm{\rho}^{(\ell)}\bm \phi\bm{\widetilde\Sigma}_{f}^{(\ell)}+\sum_{\ell}\bm{\rho}^{(\ell)}\bm \phi\bm{\Sigma}_{s}^{(\ell)}.
\end{align*}
We solve the above equation for $\bm{\phi}$ and $\keff$ with an inverse power iteration. For this, an iteration index is assigned to $\bm{\phi}$ and an update of the scalar flux is given by
\begin{align}\label{eq:poweriteration1}
     -\sum_{\ell,k} \bm D^{(\ell,k)}\bm{\widetilde\phi}^{n+1} \bm M^{(\ell,k)} + \sum_{\ell}\bm{\rho}^{(\ell)}\bm{\widetilde\phi}^{n+1}\left(\bm{\Sigma}_t^{(\ell)}-\bm{\Sigma}_{s}^{(\ell)}\right) = \sum_{\ell} \bm{\rho}^{(\ell)}\bm{\phi}^{n}\bm{\widetilde\Sigma}_{f}^{(\ell)}.
\end{align}
The iteration method to determine the eigenvalue $k$ then reads
\begin{enumerate}
    \item Start with initial guess $\bm{\phi}^0_g$ with $g=1,\cdots, G$.
    \item Compute $\bm{\widetilde\phi}^{n+1}$ from \eqref{eq:poweriteration1}.
    \item Set $k_{n+1} = \Vert \bm{\widetilde\phi}^{n+1} \Vert$ and $\bm{\phi}^{n+1} = \bm{\widetilde\phi}^{n+1}/ k_{n+1}$
    \item If $|k_{n+1}-k_n|\leq\varepsilon$ stop, else set $n\leftarrow n+1$ and repeat from step 2.
\end{enumerate}
Note that the chosen norm $\Vert \cdot \Vert$ denotes the Frobenius norm.

\subsection{Dynamical low-rank approximation}\label{sec:DLRAIntro}

In the following, the dynamical low-rank approximation \cite{KochLubich07} will be reviewed. We will employ this method to derive a memory efficient iteration scheme. DLRA is commonly derived for time dependent problems of the form
\begin{align}\label{eq:advectionSystem}
\dot{\mathbf u}(t) = \mathbf{F}(\mathbf u(t)),
\end{align}
where $\bm{u}\in\mathbb{R}^{N\times M}$ and $\bm{F}:\mathbb{R}^{N\times M}\rightarrow\mathbb{R}^{N\times M}$. To reduce computational complexity as well as memory requirements, DLRA represents and evolves the solution to \eqref{eq:advectionSystem} on a manifold of rank $r$ functions. In our case, the solution is represented by an SVD-like decomposition of the form
\begin{align}\label{eq:rankrsol}
\mathbf u(t)\approx\mathbf{X}(t)\mathbf{S}(t)\mathbf{W}(t)^T,
\end{align}
where $\mathbf{X}\in\mathbb{R}^{N\times r}$ and $\mathbf{W}\in\mathbb{R}^{M\times r}$ can be interpreted as basis matrices with orthogonal columns and $\bm S\in\mathbb{R}^{r\times r}$ is a (not necessarily diagonal) coefficient matrix. We denote the set of matrices of the form \eqref{eq:rankrsol}, i.e., the set of rank $r$ matrices, as $\mathcal{M}_r$. In order to derive evolution equations for the two basis matrices and their corresponding coefficient matrix, we wish to find $\mathbf u_r\in\mathcal{M}_r$ which fulfills
\begin{align}\label{eq:DLRproblem}
\dot{\mathbf u}_r(t)\in T_{ \mathbf u_r(t)}\mathcal{M}_r \qquad \text{such that} \qquad \left\Vert \dot{\mathbf u}_r(t)-\mathbf{F}(\mathbf u(t)) \right\Vert = \text{min}.
\end{align}
Here, $T_{\mathbf u_r(t)}\mathcal{M}_r$ denotes the tangent space of $\mathcal{M}_r$ at $\mathbf u_r(t)$. Then, following \cite[Proposition~2.1]{KochLubich07}, the condition \eqref{eq:DLRproblem} leads to update equations 
\begin{subequations}\label{eq:equationsFactors}
\begin{align}
    \dot{\bm{S}} =& \bm{X}^T\mathbf{F}(\mathbf u(t))\bm{W}, \\
    \dot{\bm{X}} =& (\bm I - \bm{X}\bm{X}^T)\mathbf{F}(\mathbf u(t))\bm{W}\bm{S}^{-1}, \\
    \dot{\bm{W}} =& (\bm I - \bm{W}\bm{W}^T)\mathbf{F}(\mathbf u(t))^T\bm{X}\bm{S}^{-T}.
\end{align}
\end{subequations}
Due to the inverse coefficient matrix on the right-hand side of the above equations, this formulation is not robust under small eigenvalues. Robust integrators for the low-rank factors are the \textit{projector-splitting integrator} \cite{LubichOseledets} as well as the \textit{unconventional integrator} \cite{CeL21}. In this work, we make use of the latter, however a derivation for the projector-splitting integrator is possible as well. The unconventional integrator updates the solution from time $t_0$ to $t_1$ via
\begin{enumerate}
    \item \textbf{$K$-step}: Update $\mathbf X^{0}$ to $\mathbf X^{1}$ via
    \begin{align}
        \dot{\mathbf K}(t) &= \mathbf{F}(\mathbf{K}(t)\mathbf{W}^{0,T})\mathbf{W}^0, \qquad \mathbf K(t_0) = \mathbf{X}^0\mathbf{S}^0.\label{eq:KStepSemiDiscreteUI}
    \end{align}
Determine $\mathbf X^1$ with $\mathbf K(t_1) = \mathbf X^1 \mathbf R$ and store $\mathbf M = \mathbf X^{1,T}\mathbf X^0$.
\item \textbf{$L$-step}: Update $\mathbf W^0$ to $\mathbf W^1$ via
\begin{align}
\dot{\mathbf L}(t) &= \mathbf{X}^{0,T}\mathbf{F}(\mathbf{X}^0\mathbf{L}(t)), \qquad \mathbf L(t_0) = \mathbf{S}^0 \mathbf{W}^{0,T}.\label{eq:LStepSemiDiscreteUI}
\end{align}
Determine $\mathbf W^1$ with $\mathbf L^1 = \mathbf W^1\mathbf{\widetilde R}$ and store $\mathbf N = \mathbf W^{1,T} \mathbf W^0$.
\item \textbf{$S$-step}: Update $\mathbf S^0$ to $\mathbf S^1$ via
\begin{align}
\dot{\mathbf S}(t) = \mathbf{X}^{1,T}\mathbf{F}(\mathbf{X}^{1}\mathbf{S}(t)\mathbf{W}^{1,T})\mathbf{W}^1, \qquad \mathbf S(t_0) &= \mathbf M\mathbf S^0 \mathbf N^T\label{eq:SStepSemiDiscreteUI}
\end{align}
and set $\mathbf S^1 = \mathbf S(t_1)$.
\end{enumerate}
The updated solution is then given as $\bm{u}(t_1) = \bm X(t_1)\bm S(t_1)\bm W(t_1)^T$. When using a forward Euler time discretization we obtain the equations
\begin{subequations}\label{eq:discreteDLRAUpdates}
\begin{align}
    \mathbf{K}^{n+1} &= \bm{K}^n + \Delta t \mathbf{F}(\mathbf{K}^n\mathbf{W}^{n,T})\mathbf{W}^n, \qquad \mathbf K^n = \mathbf{X}^n\mathbf{S}^n, \\
    \mathbf L^{n+1} &= \mathbf L^{n} + \Delta t\mathbf{X}^{n,T}\mathbf{F}(\mathbf{X}^n\mathbf{L}^n), \qquad \mathbf L^n = \mathbf{S}^n \mathbf{W}^{n,T}, \\
    \mathbf S^{n+1} &= \mathbf S^{n}+\Delta t\mathbf{X}^{n+1,T}\mathbf{F}(\mathbf{X}^{n+1}\mathbf{X}^{n+1,T}\mathbf{X}^{n}\mathbf{S}^n\mathbf{W}^{n,T}\mathbf{W}^{n+1}\mathbf{W}^{n+1,T})\mathbf{W}^{n+1}.
\end{align}
\end{subequations}
In the following, we use the above scheme to define a low-rank inverse power iteration for the $k$-eigenvalue problem.

\section{Dynamical low-rank approximation for the inverse power iteration}\label{sec:DLRAPowerIt}
\subsection{Iteration equations for low-rank factors}
In the following, we derive the DLRA evolution equations of the unconventional integrator. The equations of the matrix projector-splitting integrator take a similar form. Let us choose a rank $r$ representation for $\bm \phi^n$, which reads $\bm\phi^n \approx\bm{X}^n\bm{S}^n\bm{W}^n$. Here, $\bm X^n\in\mathbb{R}^{N_x\times r}$, $\bm S^n\in\mathbb{R}^{r\times r}$ and $\bm{W}^n\in\mathbb{R}^{G\times r}$. The main idea of our derivation is to interpret the iteration index $n$ as a pseudo-time, i.e., we can use \eqref{eq:discreteDLRAUpdates} to define an iteration on the factoring matrices. Let us first plug the rank $r$ representation into \eqref{eq:poweriteration1}. Furthermore, we define $\bm{\Sigma}^{(\ell)}:=\bm{\Sigma}_t^{(\ell)}-\bm{\Sigma}_{s}^{(\ell)}$. Then we have
\begin{align}\label{eq:iterationLowRankInput}
     -\sum_{\ell,k} \bm D^{(\ell,k)}\bm{\widetilde\phi}^{n+1} \bm M^{(\ell,k)} + \sum_{\ell}\bm{\rho}^{(\ell)}\bm{\widetilde\phi}^{n+1}\bm{\Sigma}^{(\ell)} = \sum_{\ell} \bm{\rho}^{(\ell)}\bm{X}^n\bm{S}^n\bm{W}^{n,T}\bm{\widetilde\Sigma}_{f}^{(\ell)}.
\end{align}
This form simplifies the presentation of the following $K$, $L$ and $S$-step derivations:

\textbf{K-step:} We define $\bm{K}^{n} = \bm{X}^n\bm{S}^n$, set $\bm{\widetilde\phi}^{n+1} = \bm{K}^{n+1}\bm{W}^{n,T}$ and multiply \eqref{eq:iterationLowRankInput} with $\bm{W}^n$ from the right. Then, we have
\begin{align*}
     -\sum_{\ell,k} \bm D^{(\ell,k)}\bm{K}^{n+1}\bm{W}^{n,T} \bm M^{(\ell,k)}\bm{W}^n + \sum_{\ell}\bm{\rho}^{(\ell)}\bm{K}^{n+1}\bm{W}^{n,T}\bm{\Sigma}^{(\ell)}\bm{W}^n = \sum_{\ell} \bm{\rho}^{(\ell)}\bm{K}^{n}\bm{W}^{n,T}\bm{\widetilde\Sigma}_{f}^{(\ell)}\bm{W}^n.
\end{align*}
The terms
\begin{align*}
    \bm{\widehat\Sigma}_{f,n}^{(\ell)} = \bm{W}^{n,T}\bm{\widetilde\Sigma}_{f}^{(\ell)}\bm{W}^n, \qquad \bm{\widehat M}^{(\ell,k)}_n=\bm{W}^{n,T} \bm M^{(\ell,k)}\bm{W}^n,\qquad\bm{\widehat\Sigma}_n^{(\ell)}=\bm{W}^{n,T}\bm{\Sigma}^{(\ell)}\bm{W}^n
\end{align*}
can be computed in $O(r^2\cdot G^2)$ operations. With the above definitions, the $K$-step reads
\begin{align}\label{eq:KStep}
    -\sum_{\ell,k} \bm D^{(\ell,k)}\bm{K}^{n+1}\bm{\widehat M}^{(\ell,k)}_n+\sum_{\ell}\bm{\rho}^{(\ell)}\bm{K}^{n+1}\bm{\widehat\Sigma}^{(\ell)}_n = \sum_{\ell} \bm{\rho}^{(\ell)}\bm{K}^{n}\bm{\widehat\Sigma}_{f,n}^{(\ell)}.
\end{align}
\textbf{L-step:} We define $\bm{L}^{n} = \bm{S}^n\bm{W}^{n,T}$, set $\bm{\widetilde\phi}^{n+1} = \bm{X}^{n}\bm{L}^{n+1}$ and multiply \eqref{eq:iterationLowRankInput} with $\bm{X}^{n,T}$ from the left. Then, we get
\begin{align*}
    -\sum_{\ell,k} \bm{X}^{n,T}\bm D^{(\ell,k)}\bm{X}^n\bm{L}^{n+1} \bm M^{(\ell,k)}+\sum_{\ell}\bm{X}^{n,T}\bm{\rho}^{(\ell)}\bm{X}^n\bm{L}^{n+1}\bm{\Sigma}^{(\ell)} = \sum_{\ell} \bm{X}^{n,T}\bm{\rho}^{(\ell)}\bm{X}^{n}\bm{L}^{n}\bm{\widetilde\Sigma}_{f}^{(\ell)}.
\end{align*}
The terms
\begin{align*}
    &\bm{\widehat \rho}^{(\ell)}_n = \bm{X}^{n,T}\bm{\rho}^{(\ell)}\bm{X}^{n}, \qquad \bm{\widehat D}^{(\ell,k)}_n=\bm{X}^{n,T}\bm D^{(\ell,k)}\bm{X}^n
\end{align*}
can be computed in $O(r^2\cdot N_x^2)$ operations. With the above definitions, the $L$-step reads
\begin{align}\label{eq:LStep}
    -\sum_{\ell,k} \bm{\widehat D}^{(\ell,k)}_n\bm{L}^{n+1} \bm M^{(\ell,k)} +\sum_{\ell}\bm{\widehat \rho}^{(\ell)}_{n}\bm{L}^{n+1}\bm{\Sigma}^{(\ell)} = \sum_{\ell} \bm{\widehat \rho}^{(\ell)}\bm{L}^{n}\bm{\widetilde\Sigma}_{f}^{(\ell)}.
\end{align}
\textbf{S-step:} We define $\bm{S} = \bm{X}^{n+1,T}\bm{X}^{n}\bm{S}^n\bm{W}^{n,T}\bm{W}^{n+1}$, set $\bm{\widetilde\phi}^{n+1} = \bm{X}^{n+1}\bm{S}^{n+1}\bm{W}^{n+1,T}$ and multiply \eqref{eq:iterationLowRankInput} with $\bm{X}^{n+1,T}$ from the left and $\bm{W}^{n+1}$ from the right. Then, we get
\begin{align*}
    -\sum_{\ell,k}& \bm{X}^{n+1,T}\bm D^{(\ell,k)}\bm{X}^{n+1}\bm{S}^{n+1}\bm{W}^{n+1,T} \bm{M}^{(\ell,k)}\bm{W}^{n+1}\\
    &+\sum_{\ell}\bm{X}^{n+1,T}\bm{\rho}^{(\ell)}\bm{X}^{n+1}\bm{S}^{n+1}\bm{W}^{n+1,T}\bm{\Sigma}^{(\ell)}\bm{W}^{n+1}= \sum_{\ell} \bm{X}^{n+1,T}\bm{\rho}^{(\ell)}\bm{X}^{n+1}\bm{S}\bm{W}^{n+1,T}\bm{\widetilde\Sigma}_{f}^{(\ell)}\bm{W}^{n+1} .
\end{align*}
Let us reuse our above definitions, but evaluate them at iteration $n+1$ instead of $n$. Then, the $S$-step reads
\begin{align}\label{eq:SStep}
    -\sum_{\ell} \bm{\widehat D}^{(\ell)}_{n+1}\bm{S}^{n+1}\bm{\widehat M}^{(\ell)}_{n+1}+\sum_{\ell}\bm{\widehat \rho}^{(\ell)}_{n+1}\bm{S}^{n+1}\bm{\widehat\Sigma}^{(\ell)}_{n+1} = \sum_{\ell} \bm{\widehat \rho}^{(\ell)}_{n+1}\bm{S}\bm{\widehat\Sigma}_{f,n+1}^{(\ell)}.
\end{align}
\subsection{Inversions}
In the presented method, we frequently need to solve systems of the form
\begin{align*}
    -\sum_{\ell}\bm{A}^{(\ell)}\bm{x}\bm{B}^{(\ell)}+\sum_{\ell}\bm{C}^{(\ell)}\bm{x}\bm{D}^{(\ell)} = \bm{y}.
\end{align*}
Let us assume general dimensions $N$ and $M$. Then $\bm A^{(\ell)},\bm{C}^{(\ell)}\in\mathbb{R}^{N\times N}$, $\bm B^{(\ell)},\bm{D}^{(\ell)}\in\mathbb{R}^{M\times M}$, and we wish to determine $\bm x\in\mathbb{R}^{N\times M}$ for a right-hand-side $\bm{y}\in\mathbb{R}^{N\times M}$. Written in index notation, we obtain
\begin{align*}
    -\sum_{\ell}\sum_{j,\alpha}A_{ij}^{(\ell)}x_{j\alpha}B_{\alpha \beta}^{(\ell)}+\sum_{\ell}\sum_{j,\alpha}C_{ij}^{(\ell)}x_{j\alpha}D_{\alpha \beta}^{(\ell)} = y_{i\beta}.
\end{align*}
Let us rearrange this to a matrix-vector product:
\begin{align*}
    \sum_{j,\alpha}\sum_{\ell}\left(-A_{ij}^{(\ell)}B_{\alpha \beta}^{(\ell)}+C_{ij}^{(\ell)}D_{\alpha \beta}^{(\ell)}\right)\tilde{x}_{\alpha+(j-1)r} = \tilde{y}_{\beta+(i-1)r},
\end{align*}
where $\bm{\tilde x},\bm{\tilde y}\in\mathbb{R}^{N\cdot M}$ are the matrices $\bm x$ and $\bm{y}$ rearranged to vectors. Hence, we can define the matrix $\bm{\mathcal E}\in\mathbb{R}^{N\cdot M\times N\cdot M}$ with entries
\begin{align*}
    \mathcal E_{\beta+(i-1)r,\alpha+(j-1)r} = \sum_{\ell}\left(-A_{ij}^{(\ell)}B_{\alpha \beta}^{(\ell)}+C_{ij}^{(\ell)}D_{\alpha \beta}^{(\ell)}\right)
\end{align*}
and then solve the linear system $\bm{\mathcal E}\bm{\tilde x} = \bm{\tilde y}$. Note that for the $K$-step \eqref{eq:KStep}, we have $N = N_x$ and $M=r$. The $L$-step \eqref{eq:LStep} has $N = G$ and $M=r$ and the $S$-step \eqref{eq:SStep} has dimensions $N = M = r$. Opposed to the original problem, which inverts a matrix $\bm E \in \mathbb{R}^{N_x\cdot G\times N_x\cdot G}$, we now need to invert three significantly smaller subproblems (assuming $r\ll \max\{N_x,G\}$). It should be noted that many implementations of the multigroup diffusion equations solve the equations on a group-by-group iteration based on the Gauss-Seidel method. These methods require the storage of a matrix of dimensions $N_x \times N_x$ and a storage of a solution vector of size $N_x\cdot G$. Nevertheless, our implementation is able to avoid the building of matrices at each step and because it is not necessarily matrix free, it could utilize a wider variety of preconditioning strategies. 

\subsection{Algorithm}
The full algorithm takes the following form:
\begin{enumerate}
    \item \textbf{$K$-step}: Update $\mathbf X^{n}$ to $\mathbf X^{n+1}$ via
    \begin{align*}
        -\sum_{\ell,k} \bm D^{(\ell,k)}\bm{K}^{n+1}\bm{\widehat M}^{(\ell,k)}_n+\sum_{\ell}\bm{\rho}^{(\ell)}\bm{K}^{n+1}\bm{\widehat\Sigma}^{(\ell)}_n = \sum_{\ell} \bm{\rho}^{(\ell)}\bm{K}^{n}\bm{\widehat\Sigma}_{f,n}^{(\ell)}.
    \end{align*}
Determine $\mathbf X^{n+1}$ with $\mathbf K^{n+1} = \mathbf X^{n+1} \mathbf R$ and store $\mathbf N_x = \mathbf X^{n+1,T}\mathbf X^n$.
\item \textbf{$L$-step}: Update $\mathbf W^n$ to $\mathbf W^{n+1}$ via
\begin{align*}
    -\sum_{\ell,k} \bm{\widehat D}^{(\ell,k)}_n\bm{L}^{n+1} \bm M^{(\ell,k)} +\sum_{\ell}\bm{\widehat \rho}^{(\ell)}_{n}\bm{L}^{n+1}\bm{\Sigma}^{(\ell)} = \sum_{\ell} \bm{\widehat \rho}^{(\ell)}\bm{L}^{n}\bm{\widetilde\Sigma}_{f}^{(\ell)}.
\end{align*}
Determine $\mathbf W^{n+1}$ with $\mathbf L^{n+1} = \mathbf W^{n+1}\mathbf{\widetilde R}$ and store $\mathbf N_E = \mathbf W^{n+1,T} \mathbf W^n$.
\item \textbf{$S$-step}: Update $\mathbf S^n$ to $\mathbf S^{n+1}$ via
\begin{align}\label{eq:SStep}
    -\sum_{\ell} \bm{\widehat D}^{(\ell,k)}_{n+1}\bm{\widetilde S}^{n+1}\bm{\widehat M}^{(\ell,k)}_{n+1}+\sum_{\ell}\bm{\widehat \rho}^{(\ell)}_{n+1}\bm{\widetilde S}^{n+1}\bm{\widehat\Sigma}^{(\ell)}_{n+1} = \sum_{\ell} \bm{\widehat \rho}^{(\ell)}_{n+1}\bm{S}\bm{\widehat\Sigma}_{f,n+1}^{(\ell)}.
\end{align}
with $\bm S = \bm{N}_x\bm{S}^n\bm{N}_E^T$. Set $k^{n+1} = \Vert \bm{\widetilde S}^{n+1}\Vert$ and $\bm{S}^{n+1} = \bm{\widetilde S}^{n+1}/k^{n+1}$.
\item If $|k_{n+1}-k_n|\leq\varepsilon$ stop, else set $n\leftarrow n+1$ and repeat.
\end{enumerate}
An extension of this algorithm to rank adaptivity according to \cite{ceruti2021rank} is straight forward. We state the rank adaptive algorithm here, even though our numerical computations have all been done with the fixed rank integrator.
\begin{enumerate}
    \item \textbf{$K$-step}: Update $\mathbf X^{n}\in\mathbb{R}^{N_x\times r_n}$ to $\mathbf X^{n+1}\in\mathbb{R}^{N_x\times 2 r_n}$ via
    \begin{align*}
        -\sum_{\ell,k} \bm D^{(\ell,k)}\bm{K}^{n+1}\bm{\widehat M}^{(\ell,k)}_n+\sum_{\ell}\bm{\rho}^{(\ell)}\bm{K}^{n+1}\bm{\widehat\Sigma}^{(\ell)}_n = \sum_{\ell} \bm{\rho}^{(\ell)}\bm{K}^{n}\bm{\widehat\Sigma}_{f,n}^{(\ell)}.
    \end{align*}
Determine $\mathbf X^{n+1}$ with $[\mathbf K^{n+1}, \bm{X}^n] = \mathbf X^{n+1} \mathbf R$ and store $\mathbf N_x = \mathbf X^{n+1,T}\mathbf X^n\in\mathbb{R}^{2 r_n\times r_n}$.
\item \textbf{$L$-step}: Update $\mathbf W^n\in\mathbb{R}^{G\times r_n}$ to $\mathbf W^{n+1}\in\mathbb{R}^{G\times 2 r_n}$ via
\begin{align*}
    -\sum_{\ell,k} \bm{\widehat D}^{(\ell,k)}_n\bm{L}^{n+1} \bm M^{(\ell,k)} +\sum_{\ell}\bm{\widehat \rho}^{(\ell)}_{n}\bm{L}^{n+1}\bm{\Sigma}^{(\ell)} = \sum_{\ell} \bm{\widehat \rho}^{(\ell)}\bm{L}^{n}\bm{\widetilde\Sigma}_{f}^{(\ell)}.
\end{align*}
Determine $\mathbf W^{n+1}$ with $[\mathbf L^{n+1}, \mathbf W^n] = \mathbf W^{n+1}\mathbf{\widetilde R}$ and store $\mathbf N_E = \mathbf W^{n+1,T} \mathbf W^n\in\mathbb{R}^{2 r_n\times r_n}$.
\item \textbf{$S$-step}: Update $\mathbf S^n\in\mathbb{R}^{r_n\times r_n}$ to $\widehat{\mathbf S}^{n+1}\in\mathbb{R}^{2 r_n\times 2 r_n}$ via
\begin{align}\label{eq:SStep}
    -\sum_{\ell,k} \bm{\widehat D}^{(\ell,k)}_{n+1}\bm{\widehat S}^{n+1}\bm{\widehat M}^{(\ell,k)}_{n+1}+\sum_{\ell}\bm{\widehat \rho}^{(\ell)}_{n+1}\bm{\widetilde S}^{n+1}\bm{\widehat\Sigma}^{(\ell)}_{n+1} = \sum_{\ell} \bm{\widehat \rho}^{(\ell)}_{n+1}\bm{S}\bm{\widehat\Sigma}_{f,n+1}^{(\ell)}.
\end{align}
with $\bm S = \bm{N}_x\bm{S}^n\bm{N}_E^T\in\mathbb{R}^{2 r_n\times 2 r_n}$. 
\item \textbf{Truncation}:
		Determine the SVD $\; \bm{\widehat{S}}^{n+1}= \bm{\widehat{P}}  \bm{\widehat{\Sigma}} \bm{\widehat{Q}}^\top$ where $\bm{\widehat{\Sigma}}=\text{diag}(\sigma_j)$. For a given tolerance~$\vartheta$, choose the new rank $r_{n+1}\le 2r_n$ such that
		$$
		\biggl(\ \sum_{j=r_{n+1}+1}^{2r} \sigma_j^2 \biggr)^{1/2} \le \vartheta.
		$$
		Compute the new factors for the approximation of $\phi^{n+1}$ as follows:
Let $\bm{S}_1$ be the $r_1\times r_1$ diagonal matrix with the $r_1$ largest singular values and let $\bm{P}_1\in \mathbb{R}^{2r\times r_1}$ and $\bm{Q}_1\in \mathbb{R}^{2r\times r_1}$ contain the first $r_1$ columns of $\bm{\widehat{P}}$ and $\bm{\widehat{Q}}$, respectively. Finally, set $\bm{U}_1 = \bm{\widehat{U}} \bm{P}_1\in \mathbb{R}^{m\times r_1}$ and
$\bm{V}_1 = \bm{\widehat{V}} \bm{Q}_1 \in \mathbb{R}^{n\times r_1}$.
\item Set $k^{n+1} = \Vert \bm{\widetilde S}^{n+1}\Vert$ and $\bm{S}^{n+1} = \bm{\widetilde S}^{n+1}/k^{n+1}$. If $|k_{n+1}-k_n|\leq\varepsilon$ stop, else set $n\leftarrow n+1$ and repeat.
\end{enumerate}
\section{Convergence in a simplified setting}\label{sec:convergence}
Let us remark the following properties of the presented iteration method: If the iteration converges to a so-called steady state, the error of the maximal eigenvalue depends on the low-rank structure of the steady state solution. I.e., the accuracy of method depends on the low-rank structure of the full problem.
To investigate convergence, we investigate a simplified setting for the power iteration. Assume that we wish to determine $\lambda$ such that $\bm{A}\bm{\phi}\bm{B} = \lambda \bm C \bm{\phi} \bm D$, i.e.,
\begin{align}\label{eq:simplifiedProblem}
    \bm{\phi} = \lambda \bm{A}^{-1}\bm{C} \bm{\phi} \bm{D}\bm{B}^{-1}.
\end{align}
Then, the power iteration for the full problem becomes
\begin{align*}
    \bm{\phi}^{n+1} = \frac{\bm{A}^{-1}\bm{C} \bm{\phi}^n \bm{D}\bm{B}^{-1}}{\Vert \bm{A}^{-1}\bm{C} \bm{\phi}^{n} \bm{D}\bm{B}^{-1}\Vert} = \frac{\bm{\widehat C} \bm{\phi}^n \bm{\widehat D}}{\Vert\bm{\widehat C} \bm{\phi}^{n} \bm{\widehat D}\Vert}
\end{align*}
where $\bm{\widehat C} := \bm{A}^{-1}\bm{C}$ and $\bm{\widehat D} := \bm{D}\bm{B}^{-1}$. Let us assume that $\bm{\widehat C} = \bm{V}\bm{\Lambda} \bm{V}^{-1}$ and $\bm{\widehat D} = \bm{U}^{-1}\bm{\Sigma}\bm{U}$. In this case, our maximal eigenvalue is of rank $1$, namely $\bm{v}_1\bm{u}_1^T$. Then, we write our initial iterate as
\begin{align*}
    \phi^{0}_{jg} = \sum_{\ell,k} V_{j\ell}\alpha_{\ell k} U_{k g}
\end{align*}
or in terms of matrices $\bm{\phi}^{n} = \bm{V}\bm{\alpha}^{n} \bm{U}$. Plugging this into the iteration scheme yields
\begin{align*}
    \bm{\phi}^{1} = \frac{\bm{V}\bm{\Lambda} \bm{V}^{-1} \bm{\phi}^{0} \bm{U}^{-1}\bm{\Sigma} \bm{U}}{\Vert \bm{V}\bm{\Lambda} \bm{V}^{-1} \bm{\phi}^0 \bm{W}^{-1}\bm{\Sigma} \bm{U}\Vert} = \frac{\bm{V}\bm{\Lambda} \bm{\alpha}\bm{\Sigma} \bm{U}}{\Vert \bm{V}\bm{\Lambda} \bm{\alpha}\bm{\Sigma} \bm{U}\Vert}
\end{align*}
Applying this multiple times gives
\begin{align*}
    \bm{\phi}^{n+1} = \frac{\bm{V}\bm{\Lambda}^n \bm{\alpha}\bm{\Sigma}^n \bm{U}}{\Vert \bm{V}\bm{\Lambda}^n \bm{\alpha}\bm{\Sigma}^n \bm{U}\Vert}
\end{align*}
Let us collect columns of $\bm{V}$ and $\bm{U}$ in vectors $\bm{v}_i$ and $\bm{u}_i$. Furthermore, let us define $\mu_{ij} = \frac{\lambda_i\sigma_j}{\lambda_1\sigma_1}$. Then, we can rewrite the above expression as
\begin{align*}
    \bm{\phi}^{n+1} = \frac{\sum_{i,j}\lambda_i^n\sigma^n_j \alpha_{ij} \bm{v}_i \bm{u}_j^T}{\Vert \sum_{i,j}\lambda_i^n\sigma^n_j \alpha_{ij} \bm{v}_i \bm{u}_j^T\Vert} = \frac{\lambda_1^n\sigma_1^n}{\vert \lambda_1^n\sigma_1^n\vert}\frac{\sum_{i,j}\mu_{ij}^n \alpha_{ij} \bm{v}_i \bm{u}_j^T}{\Vert \sum_{i,j}\mu_{ij}^n \alpha_{ij} \bm{v}_i\bm{u}_j^T\Vert}.
\end{align*}
Since $\lim_{n\rightarrow\infty}\mu_{ij}^n = \delta_{i1}\delta_{j1}$, we have that $\lim_{n\rightarrow\infty}\phi^{(n+1)} = \bm{v}_1 \bm{u}_1^T$. Hence
\begin{align*}
    \lim_{n\rightarrow\infty} k^{n} = \Vert \bm{A}^{-1}\bm{C} \bm{\phi}^{n} \bm{D}\bm{B}^{-1}\Vert = \lambda_1\sigma_1.
\end{align*}
Now we come to the dynamical low-rank algorithm. Here we need to assume that the directions $\bm{v}_1$ and $\bm{u}_1$ lie in the initial basis matrices. Let us take a closer look at what this means for the spatial basis. For this, we collect the i-th column of $\bm{X}$ in the vector $\bm{X}_i$. We say that $\bm{v}_1$ is contained in $\bm{X}$ if for any $i\in\{1,\cdots,r\}$ we have a representation $\bm{X}_i = \sum_{j=1}^{N_x}T_{ij} \bm{v}_j$ with $\bm{T}_{i}=(T_{ij})_{j=1}^{N_x}\in\mathbb{R}^{N_x}$ and $T_{i1}\neq 0$. Hence, for $\bm{v}_1$ being contained in $\bm{X}$ we do not require $\bm{v}_i$ to be spanned by the $r$ columns of $\bm{X}$, we only require that $\bm{v}_i$ lies in the representation of any column. Throughout the proof, we assume the following
\begin{remark}\label{as:1}
We assume that the eigenvectors $\bm{u}_i$ and $\bm{v}_i$ have unit norm. Furthermore, it is assumed that $\bm{A}^{-1}\bm{C}$ as well as $\bm{D}\bm{B}^{-1}$ exist and have full rank.
\end{remark}
Moreover, we employ the following notation:
\begin{remark}
The Frobenius norm is denoted by $\Vert \cdot \Vert$, whereas the spectral norm is denoted by $\Vert\cdot\Vert_2$. Recall that for matrices $\bm Y$ and $\bm{Z}$ we have $\Vert \bm{Y}\bm{Z}\Vert \leq \Vert \bm{Y}\Vert_2\cdot\Vert\bm{Z}\Vert\leq \Vert \bm{Y}\Vert\cdot\Vert\bm{Z}\Vert$.
\end{remark}

Let us start with with a look at the $K$-step
\begin{lemma}\label{le:convergenceBasis}
Assume that we have a problem of the form \eqref{eq:simplifiedProblem} and there exist $N_x$ linear independent eigenvectors $\bm v_j$ of $\bm{A}^{-1}\bm{C}$. If the direction $\bm{v}_1$ of the maximal eigenvalue $\lambda_1$ is contained in any column of $\bm{X}^0$, then there exists an $i\leq r$ such that
\begin{align*}
    \Vert \bm{X}_i^n - \bm{v}_1\Vert\leq C\cdot |\lambda_2/\lambda_1|^{n},
\end{align*}
i.e., the eigenvector $\bm{v}_1$ lies in the column range of $\bm{X}^n$ for $n\rightarrow\infty$. Furthermore, for $j\neq i$, we have 
\begin{align}\label{eq:LemmaOrtho}
    \vert \bm{v}_1^T\bm{X}_j^n\vert\leq C\cdot |\lambda_2/\lambda_1|^{n}, \qquad\vert \bm{e}_1^T \bm{T}_j^n\vert \leq \widetilde C\cdot |\lambda_2/\lambda_1|^{n}.
\end{align}
\end{lemma}
\begin{proof}
Our derivation of the $K$-step yields
\begin{align*}
    \bm{A} \bm{K}^{n+1}\bm{W}^{n,T} \bm{B} \bm{W}^{n} = \bm{C} \bm{K}^{n}\bm{W}^{n,T} \bm{D} \bm{W}^{n},
\end{align*}
where superscripts again denote the iteration step. Let us define $\bm{\widetilde B}^n := \bm{W}^{n,T} \bm{B} \bm{W}^{n}$ and $\bm{\widetilde D}^n := \bm{W}^{n,T} \bm{D} \bm{W}^{n}$. Furthermore, we write $\bm{X}^n=\bm{V}\bm{T}^n$, where $\bm{T}^n\in\mathbb{R}^{N_x\times r}$. Then with $\bm{K}^{n+1} = \bm{X}^{n+1}\bm{\widetilde{S}}$ we have
\begin{align*}
    \bm{X}^{n+1}\bm{\widetilde{S}} = \bm{A}^{-1}\bm{C} \bm{K}^{n}\bm{\widetilde D}^n(\bm{\widetilde B}^n)^{-1} = \bm{\widehat C}\bm{X}^n\bm{S}^n \bm{\widetilde D}^{n}(\bm{\widetilde B}^{n})^{-1} = \bm{V}\bm{\Lambda} \bm{T}^n \bm{S}^n\bm{\widetilde D}^{n}(\bm{\widetilde B}^{n})^{-1}.
\end{align*}
Let us define the diagonal matrix $\bm{\Lambda}_c:=\text{diag}(1,\lambda_2/\lambda_1,\cdots,\lambda_{N_x}/\lambda_1)$, which gives
\begin{align*}
    \bm{X}^{n+1}\bm{\widetilde{S}} = \bm{V}\bm{\Lambda}_c \bm{T}^n \bm{S}^n\bm{\widetilde D}^{n}(\bm{\widetilde B}^{n})^{-1}\cdot \lambda_1.
\end{align*}
From the $S$-step as well as Remark~\ref{as:1}, we know that $\bm{S}^n\bm{\widetilde D}^{n}(\bm{\widetilde B}^{n})^{-1}$ is of rank $r$ for all $n$. Hence, the column range of $\bm{X}^{n+1}$, which we denote by $\text{range}(\bm{X}^{n+1})$ is
\begin{align*}
    \text{range}(\bm{X}^{n+1}) = \text{range}(\bm{V}\bm{\Lambda}_c \bm{T}^n) = \text{range}(\bm{V}\bm{\Lambda}_c \bm{V}^{-1}\bm{V} \bm{T}^n)= \text{range}(\bm{V}\bm{\Lambda}_c \bm{V}^{-1}\bm{X}^n) = \text{range}(\bm{V}\bm{\Lambda}_c^{n+1} \bm{T}^0).
\end{align*}
Therefore, there exist $r$ coefficient vectors $\bm{\alpha}_i\in\mathbb{R}^{N_x}$ such that with the normalization factor $\gamma_{i,n} :=\Vert \bm{V}\bm{\Lambda}_c^{n}\bm{\alpha}_i \Vert$ we have $\bm{X}_i^{n+1} = \frac{1}{\gamma_{i,n}}\bm{V}\bm{\Lambda}_c^{n+1}\bm{\alpha}_i$. Note that
\begin{align}
    \big| \gamma_{i,n} - |\alpha_{1i}| \big| \leq \left(\sum_{\ell=2}^{N_x}\left( \frac{\lambda_{\ell}}{\lambda_1}\right)^{2n}\alpha_{\ell i}^2\right)^{1/2}.
\end{align}
Hence, when assuming $\alpha_{1i} > 0$, we have
\begin{align*}
    \Vert \bm{X}_i^n - \bm{v}_1\Vert^2 =& \frac{1}{\gamma_{i,n}^2}\left\Vert \bm{V}(\bm{\Lambda}_c^{n}\bm{\alpha}_i-\bm{e}_1\gamma_{i,n}) \right\Vert^2 \leq \frac{\Vert \bm{V} \Vert_2^2}{\gamma_{i,n}^2}\left\Vert \bm{\Lambda}_c^{n}\bm{\alpha}_i-\bm{e}_1\gamma_{i,n} \right\Vert^2\\
    =& \frac{\Vert \bm{V} \Vert_2^2}{\gamma_{i,n}^2}\left((\alpha_{1i}-\gamma_{i,n})^2 + \sum_{\ell=2}^{N_x}\left( \frac{\lambda_{\ell}}{\lambda_1}\right)^{2n}\alpha_{\ell i}^2 \right) \leq C |\lambda_2/\lambda_1|^n.
\end{align*}
For $\alpha_{1i}\leq 0$ we have that $\Vert \bm{X}_i^n - (-\bm{v}_1)\Vert \leq C |\lambda_2/\lambda_1|^n$. Hence, $\bm{X}_i^n$ converges to $\pm \bm{v}_1$ and the eigenvector $\bm{v}_1$ lies in the range of $\bm{X}^n$ for $n\rightarrow\infty$.
For $j\neq i$, we have
\begin{align*}
    |\bm{v}_1^T \bm{X}_j^n| = |(\bm{v}_1-\bm{X}_i^n)^T \bm{X}_j^n| \leq \Vert \bm{v}_1 - \bm{X}_i^n\Vert \cdot\Vert \bm{X}_j^n\Vert \leq C\cdot|\lambda_2/\lambda_1|^n.
\end{align*}
Furthermore, since we assume that eigenvectors are linearly independent, we have
\begin{align*}
    |\bm{e}_1^T \bm{T}_j^n| = |(\bm{e}_1-\bm{T}_i^n)^T \bm{T}_j^n| \leq \Vert \bm{V}^{-1}\bm{V}(\bm{e}_1 - \bm{T}_i^n)\Vert \cdot\Vert \bm{T}_j^n\Vert \leq \frac{C}{\Vert \bm{V}^{-1}\Vert_2}\cdot|\lambda_2/\lambda_1|^n.
\end{align*}
\end{proof}
\begin{remark}\label{rem:basisW}
The same holds true for $\bm{W}^n$: If $\bm{W}^0$ contains $\bm{u}_1$, this direction lies in the range of $\bm{W}^n$ for $n\rightarrow\infty$. The proof is straightforward and simply applies the above derivation to the $L$-step.
\end{remark}
Now, we know that in the limit, the eigenvectors $\bm{v}_1$ and $\bm{u}_1$ form one of the columns of $\bm{X}$ and $\bm{W}$, respectively. If we have $\bm{X}_i = \bm{v}_1$ and $\bm{W}_{\ell} = \bm{u}_1$ it remains to show that $S_{jk}^n \rightarrow \delta_{ji}\delta_{k\ell}$, such that $\bm \phi^n\rightarrow \lambda_1 \bm{v}_1\bm{u}_1^T \sigma_1$ as $n\rightarrow\infty$. For this, we need to take a closer look at the $S$-step, which reads
\begin{align*}
    \bm{X}^{n+1,T}\bm{A} \bm{X}^{n+1}\bm{\widetilde{S}}^{n+1}\bm{W}^{n+1,T} \bm{B} \bm{W}^{n+1} = \bm{X}^{n+1,T}\bm{C} \bm{X}^{n+1} \bm{M} \bm{S}^{n}\bm{N}^{T}\bm{W}^{n+1,T} \bm{D} \bm{W}^{n+1}
\end{align*}
Therefore,
\begin{align*}
    \bm{\widetilde{S}}^{n+1} = (\bm{\widetilde{A}}^{n+1})^{-1}\bm{\widetilde{C}}^{n+1} \bm{M} \bm{S}^{n}\bm{N}^{T}\bm{\widetilde{D}}^{n+1}(\bm{\widetilde{B}}^{n+1})^{-1}.
\end{align*}
Let us represent our basis in space and energy as $\bm{X}^n = \bm{V}\bm{T}^n_x$ and $\bm{W}^n = \bm{U}\bm{T}^n_e$. Then, the $S$-step can be rewritten in the following way:
\begin{lemma}
With $\bm{\Delta}^{(x)} := \bm{I} - \bm{T}_x\bm{T}_x^T\in\mathbb{R}^{N_x\times N_x}$ and $\bm{\Delta}^{(e)} := \bm{I} - \bm{T}_e\bm{T}_e^T\in\mathbb{R}^{G\times G}$, let us define
\begin{align*}
    \bm{\mathcal E}_x^{(n+1)} :=& (\bm{\widetilde{A}}^{n+1})^{-1}\bm{X}^{n+1,T}\bm{A}\bm{V} \bm{\Delta}^{(x)}\bm{\Lambda}_c \bm{T}_x^{n+1}, \\
    \bm{\mathcal E}_e^{(n+1)} :=&   \bm{T}_e^{n+1,T}\bm{\Sigma}_c\bm{\Delta}^{(e)}\bm{U}^T\bm{B}\bm{W}^{n+1}(\bm{\widetilde{B}}^{n+1})^{-1}.
\end{align*}
Then, the $S$-step takes the form
\begin{align}\label{eq:SstepLemma}
   \bm{\widetilde S}^{n+1} = \lambda_1\sigma_1\left(\bm{T}_x^{n+1,T}\bm{\Lambda}_c \bm{T}_x^{n+1} + \bm{\mathcal E}_x^{(n+1)}\right)\bm{M}\bm{S}^n\bm{N}^T\left(\bm{\mathcal E}_e^{(n+1)}+\bm{T}_e^{n+1,T}\bm{\Sigma}_c \bm{T}_e^{n+1}\right).
\end{align}
\end{lemma}
\begin{proof}
Let us show the derivation for the spatial part only and drop the $x$ index in the following. We know that $\bm{A}^{-1}\bm{C}\bm{V} = \bm{V}\bm{\Lambda}$, hence $\bm{C}\bm{V} = \bm{A}\bm{V}\bm{\Lambda}$. Then, the matrix $\bm{\widetilde{C}}$ becomes
\begin{align*}
    \bm{\widetilde{C}}^{n+1} = \bm{X}^{n+1,T}\bm{C} \bm{X}^{n+1} = \bm{T}^{n+1,T}\bm{V}^T\bm{C}\bm{V}\bm{T}^{n+1} = \bm{T}^{n+1,T}\bm{V}^T\bm{A}\bm{V}\bm{\Lambda}\bm{T}^{n+1}.
\end{align*}
Let $\bm{\Delta} := \bm{I} - \bm{T}^{n+1}\bm{T}^{n+1,T}\in\mathbb{R}^{N_x\times N_x}$. Then, with $\bm{I} = \bm{T}^{n+1}\bm{T}^{n+1,T}+\bm{\Delta}$ we have
\begin{align*}
    \bm{\widetilde{C}}^{n+1} =& \bm{T}^{n+1,T}\bm{V}^T\bm{A}\bm{V} (\bm{T}^{n+1} \bm{T}^{n+1,T}+\bm{\Delta})\bm{\Lambda} \bm{T}^{n+1}\\
    =& \bm{\widetilde{A}}^{n+1}\bm{T}^{n+1,T}\bm{\Lambda} \bm{T}^{n+1} + \bm{T}^{n+1,T}\bm{V}^T\bm{A}\bm{V} \bm{\Delta}\bm{\Lambda} \bm{T}^{n+1}.
\end{align*}
Therefore,
\begin{align}\label{eq:SStepInProof}
    (\bm{\widetilde{A}}^{n+1})^{-1}\bm{\widetilde{C}}^{n+1} &= \bm{T}^{n+1,T}\bm{\Lambda} \bm{T}^{n+1} + (\bm{\widetilde{A}}^{n+1})^{-1}\bm{X}^{n+1,T}\bm{A}\bm{V} \bm{\Delta}\bm{\Lambda} \bm{T}^{n+1} \nonumber\\
    &= \lambda_1\left(\bm{T}^{n+1,T}\bm{\Lambda}_c \bm{T}^{n+1} + (\bm{\widetilde{A}}^{n+1})^{-1}\bm{X}^{n+1,T}\bm{A}\bm{V} \bm{\Delta}\bm{\Lambda}_c \bm{T}^{n+1}\right)\nonumber\\
    &=: \lambda_1 \left(\bm{T}^{n+1,T}\bm{\Lambda}_c \bm{T}^{n+1}+\bm{\mathcal E}_x^{(n+1)}\right).
\end{align}
The energy parts can be derived analogously, which yields the $S$-step \eqref{eq:SstepLemma}.
\end{proof}
Now we have all building block to show convergence:
\begin{theorem}
The inverse DLRA power iteration scheme as proposed in Section~\ref{sec:DLRAPowerIt} converges for a problem of the form \eqref{eq:simplifiedProblem} to the eigenvector corresponding to the maximal eigenvalue $\lambda_1\sigma_1$. Moreover, we have
\begin{align*}
     \left|\Vert \bm{\phi}^n \Vert - \lambda_1\sigma_1 \right| \leq  C_1 \left(\frac{\lambda_2}{\lambda_1}\right)^{n} + C_2\left(\frac{\sigma_2}{\sigma_1}\right)^{n}.
\end{align*}
\end{theorem}
\begin{proof}
Let us investigate the spatial part of the $S$-step \eqref{eq:SstepLemma}, namely \eqref{eq:SStepInProof} and again conclude the terms for the energy part. We start by defining $\bm{P}_i := \bm{e}_i \bm{e}_i^T$ and $\bm{P}^{\perp}_i := \bm{I}-\bm{e}_i \bm{e}_i^T$. The idea of this proof is to show that $\bm{P}_i^{\perp}\bm{S}$ and $\bm{S}(\bm{P}_{\ell}^{\perp})^T \rightarrow 0$ as $n\rightarrow\infty$. Let us start by noting that
\begin{align}\label{eq:PPerpTerm}
    \bm{P}^{\perp}_i\bm{T}^{n+1,T}\bm{\Lambda}_c \bm{T}^{n+1}\bm{M}\bm{S}^n = \bm{P}^{\perp}_i\bm{T}^{n+1,T}(\bm{P}_1+\bm{P}^{\perp}_1)\bm{\Lambda}_c(\bm{P}_1+\bm{P}^{\perp}_1) \bm{T}^{n+1}(\bm{P}_i+\bm{P}^{\perp}_i)\bm{M}\bm{S}^n.
\end{align}
With \eqref{eq:LemmaOrtho}, we have $\Vert \bm{P}^{\perp}_i \bm{T}^{n+1,T}\bm{P}_1\Vert\leq C (\lambda_2/\lambda_1)^{n+1}$ and $\Vert \bm{P}_1^{\perp} \bm{T}^{n+1}\bm{P}_i\Vert\leq C (\lambda_2/\lambda_1)^{n+1}$, hence
\begin{align*}
    \Vert \bm{P}^{\perp}_i \bm{T}^{n+1,T}\bm{P}_1\bm{\Lambda}_c \bm{T}^{n+1}\bm{M}\bm{S}^n \Vert &\leq \widetilde C (\lambda_2/\lambda_1)^{n+1},\\
    \Vert \bm{T}^{n+1,T}\bm{\Lambda}_c \bm{P}^{\perp}_1 \bm{T}^{n+1}\bm{P}_i \bm{M}\bm{S}^n \Vert &\leq \widetilde{C} (\lambda_2/\lambda_1)^{n+1}.
\end{align*}
Together with $\bm{P}_1^{\perp}\bm{\Lambda}_c \bm{P}_1 = 0$, \eqref{eq:PPerpTerm} becomes
\begin{align*}
    \Vert \bm{P}^{\perp}_i\bm{T}^{n+1,T}\bm{\Lambda}_c \bm{T}^{n+1}\bm{M}\bm{S}^n \Vert \leq \Vert \bm{P}^{\perp}_i \bm{T}^{n+1,T}\bm{P}^{\perp}_1\bm{\Lambda}_c \bm{P}^{\perp}_1 \bm{T}^{n+1}\bm{P}^{\perp}_i\bm{M}\bm{S}^n\Vert + C (\lambda_2/\lambda_1)^{n+1}.
\end{align*}
Let us observe that
\begin{align*}
    \Vert \bm{P}_i^{\perp}\bm{M} \bm{P}_i\Vert =& \Vert \bm{P}_i^{\perp}\bm{T}^{n+1,T}\bm{V}^T\bm{V} \bm{T}^n \bm{P}_i\Vert \\
    =& \Vert \bm{P}_i^{\perp}\bm{T}^{n+1,T}(\bm{V}(\bm{P}_1+\bm{P}_1^{\perp}))^T\bm{V}(\bm{P}_1+\bm{P}_1^{\perp}) \bm{T}^n \bm{P}_i\Vert \\
    \leq& \Vert \bm{P}_i^{\perp}\bm{T}^{n+1,T}(\bm{V} \bm{P}_1^{\perp})^T\bm{V} \bm{P}_1 \bm{T}^n \bm{P}_i\Vert + C (\lambda_2/\lambda_1)^{n+1} = C (\lambda_2/\lambda_1)^{n+1},
\end{align*}
where in the last step we used $(\bm{V}\bm{P}_1^{\perp})^T\bm{V}\bm{P}_1 = \bm 0$. Hence, with $\Lambda_0 := \text{diag}(0,\lambda_2/\lambda_1,\cdots,\lambda_{N_x}/\lambda_1)$, we have
\begin{align*}
    \Vert \bm{P}^{\perp}_i\bm{T}^{n+1,T}\bm{\Lambda}_c \bm{T}^{n+1}\bm{M}\bm{S}^n\Vert \leq& \Vert \bm{P}^{\perp}_i \bm{T}^{n+1,T}\bm{P}^{\perp}_1\bm{\Lambda}_c \bm{P}^{\perp}_1 \bm{T}^{n+1}\bm{P}^{\perp}_i\bm{M}\bm{S}^n\Vert+ C (\lambda_2/\lambda_1)^{n+1} \\
    \leq& \Vert \bm{P}^{\perp}_i \bm{T}^{n+1,T}\bm{P}^{\perp}_1\bm{\Lambda}_c \bm{P}^{\perp}_1 \bm{T}^{n+1}\bm{P}^{\perp}_i\bm{M}\bm{P}^{\perp}_i\bm{S}^n\Vert+ C (\lambda_2/\lambda_1)^{n+1}\\
    \leq& \Vert \bm{P}^{\perp}_i \bm{T}^{n+1,T}\bm{\Lambda}_0 \bm{T}^{n+1}\bm{M}\bm{P}^{\perp}_i\bm{S}^n\Vert+ C (\lambda_2/\lambda_1)^{n+1}.
\end{align*}
In the same manner we have with $\bm{\mathcal E}_{x,0}^{(n+1)} := (\bm{\widetilde{A}}^{n+1})^{-1}\bm{X}^{n+1,T}\bm{A}\bm{V} \bm{\Delta}\bm{\Lambda}_0 \bm{T}_x^{n+1}$ that
\begin{align*}
    \Vert \bm{\mathcal E}_x^{(n+1)}\bm{M}\bm{S}^n \Vert \leq \Vert \bm{\mathcal E}_{x,0}^{(n+1)}\bm{M}\bm{P}^{\perp}_i \bm{S}^n \Vert + C_2 (\lambda_2/\lambda_1)^{n+1}.
\end{align*}
\begin{align*}
   \left\Vert\bm{P}_i^{\perp}\left(\bm{T}^{n+1,T}\bm{\Lambda}_c \bm{T}^{n+1}+\bm{\mathcal E}_x^{(n+1)}\right)\bm{M}\bm{S}^n\right\Vert \leq \left\Vert\bm{P}_i^{\perp}\left(\bm{T}^{n+1,T}\bm{V}^T\bm{A}\bm{V}\bm{\Lambda}_0\bm{T}^{n+1}\right)\bm{M}\bm{P}_i^{\perp}\bm{S}^n\right\Vert + C (\lambda_2/\lambda_1)^{n+1}.
\end{align*}
All together, this gives the estimate
\begin{align*}
    &\frac{\lambda_1 \left\Vert\bm{P}_i^{\perp}\left(\bm{T}^{n+1,T}\bm{\Lambda}_c \bm{T}^{n+1}+\bm{\mathcal E}_x^{(n+1)}\right)\bm{M}\bm{S}^n\right\Vert}{\lambda_1 \left\Vert\left(\bm{T}^{n+1,T}\bm{\Lambda}_c \bm{T}^{n+1}+\bm{\mathcal E}_x^{(n+1)}\right)\bm{M}\bm{S}^n\right\Vert}\\
    \leq& \frac{\Vert \bm{P}^{\perp}_i \left(\bm{T}^{n+1,T}\bm{\Lambda}_0 \bm{T}^{n+1} +\bm{\mathcal E}_{x,0}^{(n+1)}\right)\bm{M}\bm{P}^{\perp}_i\bm{S}^n\Vert}{\left\Vert\left(\bm{T}^{n+1,T}\bm{\Lambda}_c \bm{T}^{n+1}+\bm{\mathcal E}_x^{(n+1)}\right)\bm{M}\bm{S}^n\right\Vert}+ C(\lambda_2/\lambda_1)^{n+1}\\
    \leq&\widehat{C}\cdot\Vert \bm{P}^{\perp}_i\bm{S}^n\Vert+ C(\lambda_2/\lambda_1)^{n+1}.
\end{align*}
Since by the normalization step of our scheme $\Vert \bm S^n\Vert = \Vert \bm{S}^{n,-1}\Vert = 1$, we have
\begin{align*}
    \widehat{C} :=& \frac{\left\Vert \bm{P}^{\perp}_i \left(\bm{T}^{n+1,T}\bm{\Lambda}_0 \bm{T}^{n+1} +\bm{\mathcal E}_{x,0}^{(n+1)}\right)\bm{M}\bm{S}^n\bm{S}^{n,-1}\right\Vert}{\left\Vert\left(\bm{T}^{n+1,T}\bm{\Lambda}_c \bm{T}^{n+1}+\bm{\mathcal E}_x^{(n+1)}\right)\bm{M}\bm{S}^n\right\Vert} \leq \frac{\Vert \bm{P}^{\perp}_i \Vert_2 \left\Vert \left(\bm{T}^{n+1,T}\bm{\Lambda}_0 \bm{T}^{n+1} +\bm{\mathcal E}_{x,0}^{(n+1)}\right)\bm{M}\bm{S}^n\right\Vert\cdot\Vert\bm{S}^{n,-1}\Vert}{\left\Vert\left(\bm{T}^{n+1,T}\bm{\Lambda}_c \bm{T}^{n+1}+\bm{\mathcal E}_x^{(n+1)}\right)\bm{M}\bm{S}^n\right\Vert}\\
    =& \frac{\left\Vert \left(\bm{T}^{n+1,T}\bm{\Lambda}_0 \bm{T}^{n+1} +\bm{\mathcal E}_{x,0}^{(n+1)}\right)\bm{M}\bm{S}^n\right\Vert}{\left\Vert\left(\bm{T}^{n+1,T}\bm{\Lambda}_c \bm{T}^{n+1}+\bm{\mathcal E}_x^{(n+1)}\right)\bm{M}\bm{S}^n\right\Vert}\leq \frac{\lambda_2}{\lambda_1}.
\end{align*}
Including the term $\bm{N}^{T}\bm{\widetilde{D}}^{n+1}(\bm{\widetilde{B}}^{n+1})^{-1}$ and utilizing the same arguments as above yields
\begin{align*}
    \Vert \bm{P}^{\perp}_i\bm{S}^{n+1}(\bm{P}^{\perp}_{\ell})^T\Vert \leq& \frac{\lambda_2\sigma_2}{\lambda_1\sigma_1} \Vert \bm{P}^{\perp}_i\bm{S}^n (\bm{P}^{\perp}_{\ell})^T\Vert + C \left(\frac{\lambda_2\sigma_2}{\lambda_1\sigma_1}\right)^{n+1}\\
    \Vert \bm{P}^{\perp}_i\bm{S}^{n+1}\bm{P}_{\ell}^T\Vert \leq& \frac{\lambda_2}{\lambda_1} \Vert \bm{P}^{\perp}_i\bm{S}^n \bm{P}_{\ell}^T\Vert + C \left(\frac{\lambda_2}{\lambda_1}\right)^{n+1}\\
    \Vert \bm{P}_i\bm{S}^{n+1}(\bm{P}^{\perp}_{\ell})^T\Vert \leq& \frac{\sigma_2}{\sigma_1} \Vert \bm{P}_i\bm{S}^n (\bm{P}^{\perp}_{\ell})^T\Vert + C \left(\frac{\sigma_2}{\sigma_1}\right)^{n+1}
\end{align*}
Recursively, we get
\begin{align*}
    \Vert \bm{P}^{\perp}_i\bm{S}^{n}(\bm{P}^{\perp}_{\ell})^T\Vert \leq c \left(\frac{\lambda_2\sigma_2}{\lambda_1\sigma_1}\right)^{n}, \quad \Vert \bm{P}^{\perp}_i\bm{S}^{n}\bm{P}_{\ell}^T\Vert \leq c \left(\frac{\lambda_2}{\lambda_1}\right)^{n}, \quad \Vert \bm{P}_i\bm{S}^{n}(\bm{P}^{\perp}_{\ell})^T\Vert \leq c \left(\frac{\sigma_2}{\sigma_1}\right)^{n}.
\end{align*}
Since $\Vert \bm{S}^{n}\Vert = \Vert (\bm{P}_i+\bm{P}^{\perp}_i)\bm{S}^{n}(\bm{P}_{\ell}+\bm{P}^{\perp}_{\ell})^T\Vert = 1$, we know that 
\begin{align*}
    \left|\Vert \bm{P}_i \bm{S}^{n}\bm{P}_{\ell}^T \Vert - 1 \right| \leq C_0 \left(\frac{\lambda_2\sigma_2}{\lambda_1\sigma_1}\right)^{n} +  C_1 \left(\frac{\lambda_2}{\lambda_1}\right)^{n} + C_2\left(\frac{\sigma_2}{\sigma_1}\right)^{n}.
\end{align*}
I.e., $S_{kj}^n\rightarrow \delta_{ik}\delta_{\ell j}$ and with Lemma~\ref{le:convergenceBasis} and Remark~\ref{rem:basisW} we conclude the theorem.
\end{proof}


\section{Numerical Results}\label{sec:numerical}
In the following, we investigate the proposed algorithm for different material and geometric settings. Results will be compared against the full code framework \cite{codeNeutronDiffusion}. Note that the memory requirements of this code framework do not allow the computation of a finely resolved solution, which is why we show that the DLRA approach converges to the same solution as the full problem for sufficiently large rank. The dynamical low-rank code that has been used in this work can be found in \cite{code}. Note that the material data used in this work cannot be made publicly available. To compare against a finely resolved reference solution, a DLRA solution with high rank will be used.

\subsection{Stainless-steel reflected uranium sphere}
The first problem we consider is the IEU-MET-FAST-005 criticality benchmark from the OECD/NEA suite \cite{oecd}. This problem has a sphere of 36\% enriched uranium surrounded by a neutron reflector comprised of stainless steel. The problem has an overall radius of $21.486$ cm and the uranium sphere has a radius of $13.213$ cm. The stainless steel is divided into shells with two different densities: one   with radius $1.758$ cm and the other in the remainder of the total size.

The dynamical low-rank approximation of this setting is derived according to Section~\ref{sec:DLRAPowerIt}. A spatial discretization with $400$ spatial cells is chosen. The energy domain is represented by $87$ energy groups.  
\begin{figure}[htp!]
\centering
		\centering
		\includegraphics[width=0.49\linewidth]{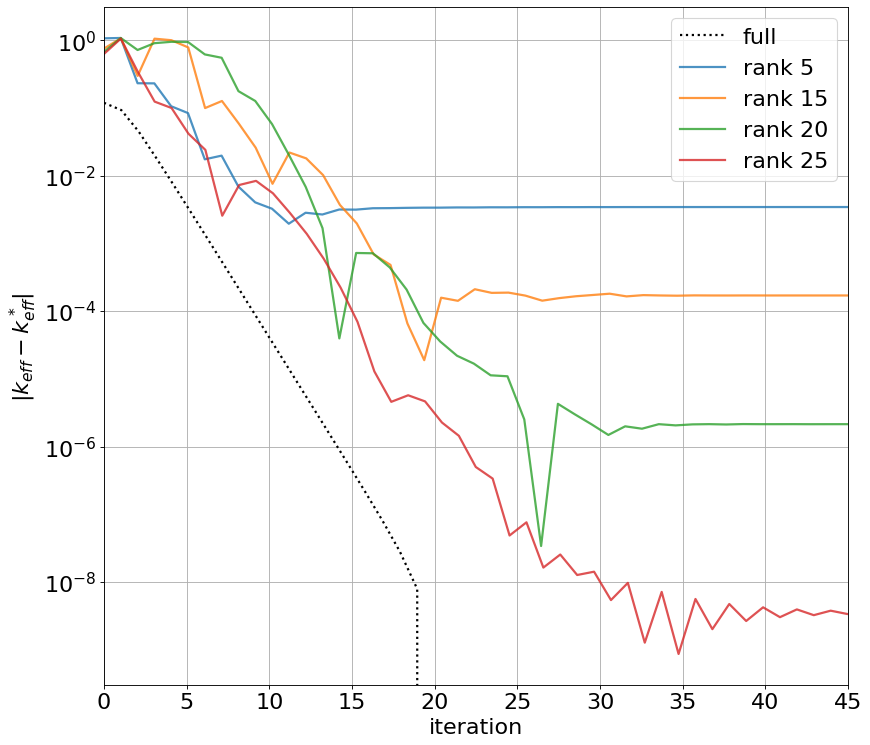}
    \caption{Convergence of the effective eigenvalue for the uranium sphere testcase. As reference effective eigenvalue $k_\mathrm{eff}^*$ for the $87$ group problem, the converged solution of the full inverse power iteration is used. The number of spatial cells is $N_x = 400$.}
	\label{fig:historyMaterialuranium_01}
\end{figure}
Figure~\ref{fig:historyMaterialuranium_01} shows the convergence of the effective eigenvalue for the full problem as well as the DLRA approximation. It is observed that the solutions for rank $20$ and $25$ show satisfactory convergence properties. Lower ranks result in unsatisfactory approximations of the effective eigenvalue. Furthermore, it should be noted that the DLRA method requires an increased number of iterations to reach a converged state. 

\begin{figure}[htp!]
\centering
    \begin{subfigure}{0.49\linewidth}
		\centering
		\includegraphics[width=\linewidth]{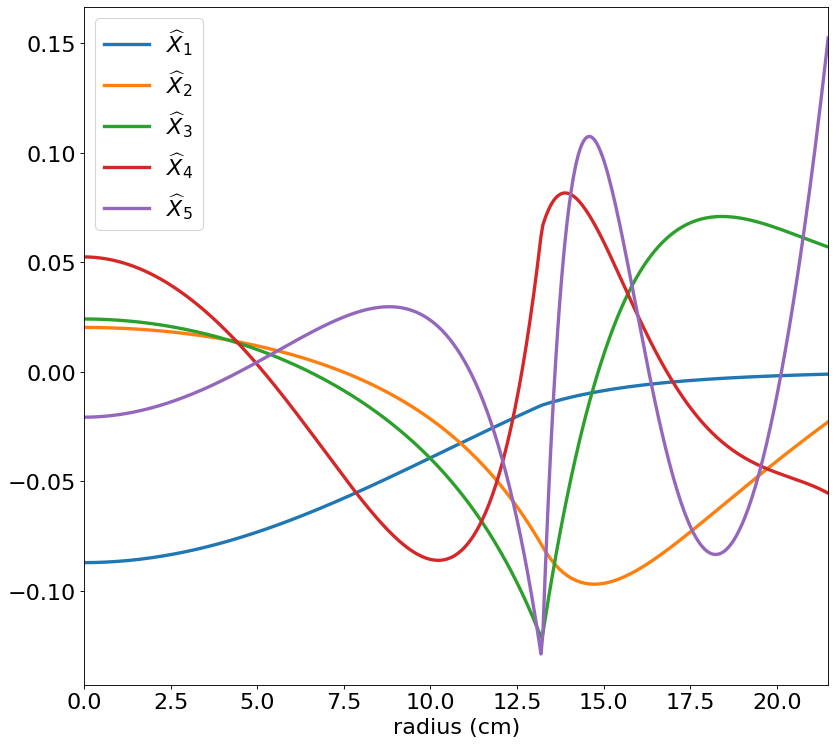}
		\caption{Basis function in radius}
		\label{fig:studyCFLPNRank15}
	\end{subfigure}
	\begin{subfigure}{0.49\linewidth}
		\centering
		\includegraphics[width=\linewidth]{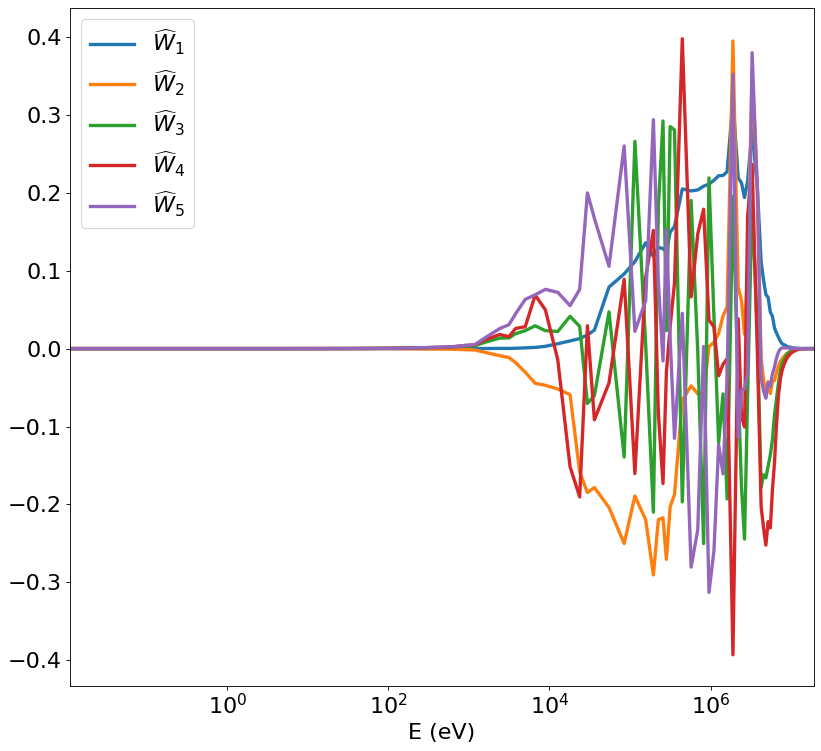}
		\caption{Basis functions in energy}
		\label{fig:studyCFLPNRank10}
	\end{subfigure}
		\begin{subfigure}{0.49\linewidth}
		\centering
		\includegraphics[width=\linewidth]{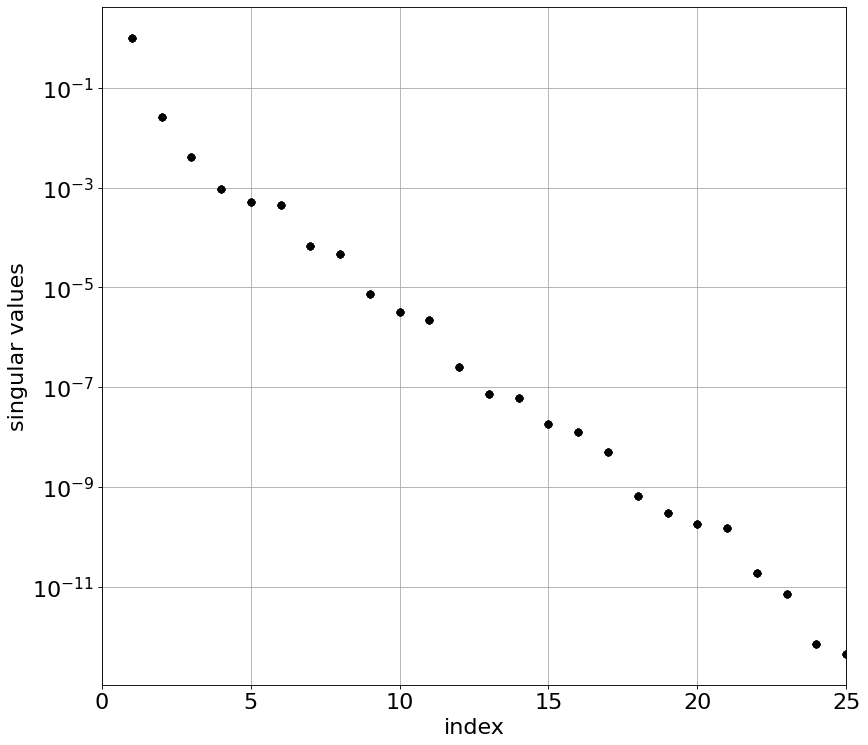}
		\caption{Singular values of the solution matrix}
		\label{fig:studyCFLPNRank10}
	\end{subfigure}
    \caption{First five DLRA basis functions in (a) radius and (b) energy as well as (c) singular values for the uranium sphere problem $N_x = 400$, $G = 87$ and rank $r=25$.}
	\label{fig:basis_uranium_01}
\end{figure}
From the resulting DLRA factorization, we can investigate the dynamics of the system. For this, we compute an SVD decomposition of the coefficient matrix $\bm{S} = \bm{U}\bm{\Sigma}\bm{V}^T$ and plot the vectors $\bm{\widehat X}_i := \bm{X}\bm{U}_i\in\mathbb{R}^{N_x}$ as well as $\bm{\widehat W}_i:= \bm{W}\bm{V}_i\in\mathbb{R}^{G}$, where for sake of representation we only look at the first five basis functions. Basis functions and the corresponding eigenvalues of the diagonal matrix $\bm{\Sigma}\in\mathbb{R}^{r\times r}$ are shown in Figure~\ref{fig:basis_uranium_01}. The basis functions encode both the spatial geometry as well as physical effects. First, the spatial basis captures the change of the background material, especially the transition from uranium to stainless steel. Second, the energy basis encodes the appearance of mostly high-energy particles. The corresponding eigenvalues decay rather slowly as their index increases, which is an indicator for the effectiveness of the method only when the chosen rank is sufficiently high.
\begin{figure}[htp!]
\centering
    \begin{subfigure}{0.49\linewidth}
		\centering
		\includegraphics[width=\linewidth]{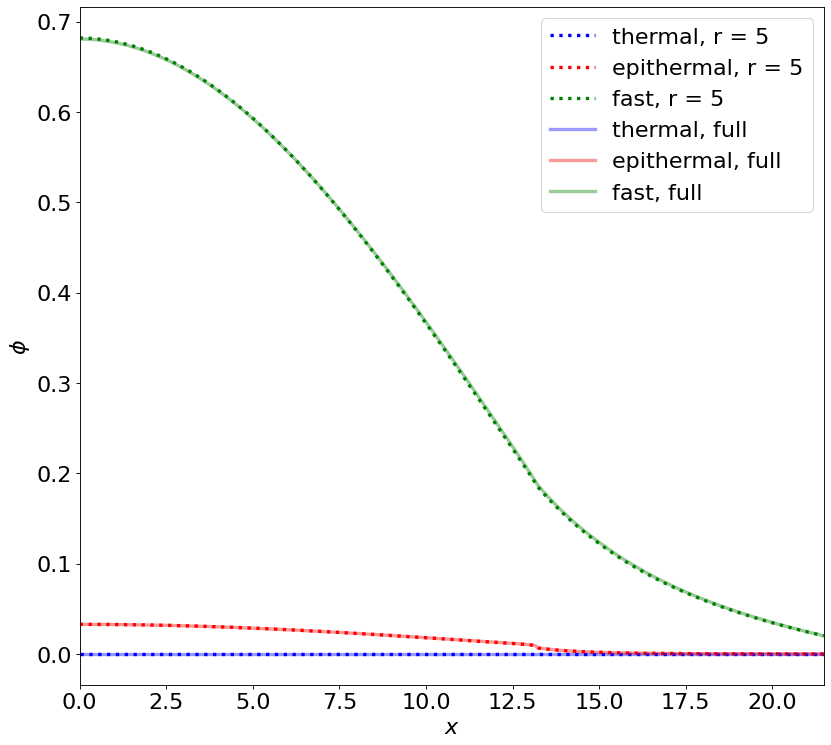}
		\caption{Thermal, epithermal and fast neutrons}
		\label{fig:studyCFLPNRank15Space}
	\end{subfigure}
	   \begin{subfigure}{0.49\linewidth}
		\centering
		\includegraphics[width=\linewidth]{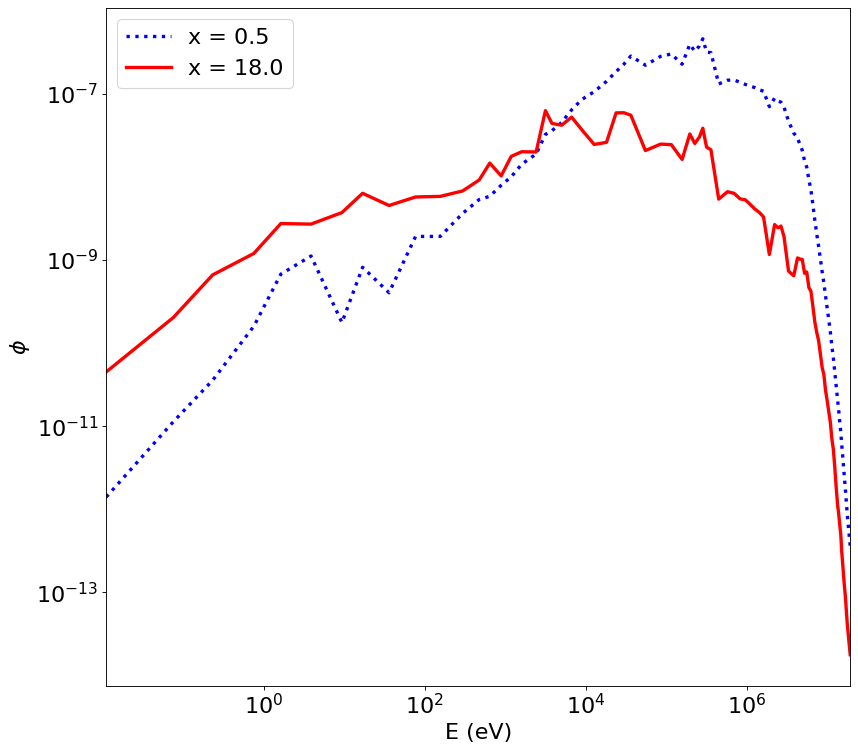}
		\caption{$\phi$ at fixed spatial positions.}
		\label{fig:studyCFLPNRank15Energy}
	\end{subfigure}
    \caption{Left: Thermal ($E\in [0,5]$), epithermal ($E\in (5,0.5\cdot 10^6]$) and fast ($E\in (0.5\cdot 10^6,\infty)$) neutrons with a rank $r=5$ approximation and $N_x = 400$, $G = 87$ as well as the full solution. Right: $\phi$ at fixed spatial positions computed with a rank $25$ approximation.}
	\label{fig:ssSols}
\end{figure}

In Figure \ref{fig:ssSols} we look further at the properties of the solution. The spatial variation in the solution integrated over different energy ranges is shown in Figure \ref{fig:studyCFLPNRank15}. In this figure the fast energy range is all groups above 0.5 MeV, epithermal is above 5 eV up to 0.5 MeV, and everything at 5 eV and below is the thermal range. For this problem we observe that the problem is dominated by fast neutrons, as is to be expected because the stainless steel reflector does not moderate neutrons particularly well. To look at the shape of the solution in energy, Figure \ref{fig:studyCFLPNRank15Space} plots $\phi$ as a function of energy at different spatial points: one near the outer radius of the sphere and one in the interior. Because $\phi_g$ in our solution is the integral over the group energy range, in this and subsequent plots we display the average value $\phi(r,E) = \phi_g(r)/\Delta E_g$ for $E$ in group $g$ where $\Delta E_g$ is the width of group $g$. From Figure \ref{fig:studyCFLPNRank15Energy} we can see that there is a significant shift in the energy spectrum at different spatial points in the system. 

\subsection{Light water reactor}
In this problem we look at the solution for a problem of a homogenized light water reactor using the SHEM 361-group energy group structure \cite{hebert2008refinement}. The problem consists of this homogenized material in a sphere of radius 79.06925 cm. We expect the solution to have more neutrons in the thermal energy range than the previous problem. Furthermore, this problem has a large number of energy groups. From Figure \ref{fig:historyMaterial-lwr} we see that the eigenvalue $\keff$ converges within 1 pcm (1 percent-mille = $10^-5$) with a rank of 7 when compared to a reference calculation with rank 25. For this problem, the memory requirements do not allow a computation with the full solver and only DLRA results are available. This indicates the usefulness of the DLRA solver. The reduced memory requirement of implementations is depicted in Figure~\ref{fig:memoryMaterial-lwr}, where we observe that the memory of the full method grows significantly faster then the DLRA method, therefore only allowing the use of $100$ spatial cells. Note that the full method sets up the full system matrix with $N_x^2\cdot G^2$ entries. A consecutive computation of submatrices for each energy group is possible and reduces memory requirements. However, this approach results in significantly increased runtimes. 

We also note that the number of iterations required is much larger for this problem. This indicates a large dominance ratio, i.e., the ratio of the fundamental eigenvalue to the first harmonic., as would be expected in a problem with a large spatial extent. 
\begin{figure}[htp!]
\centering
		\centering
		\begin{subfigure}{0.49\linewidth}
		\includegraphics[width=1.0\linewidth]{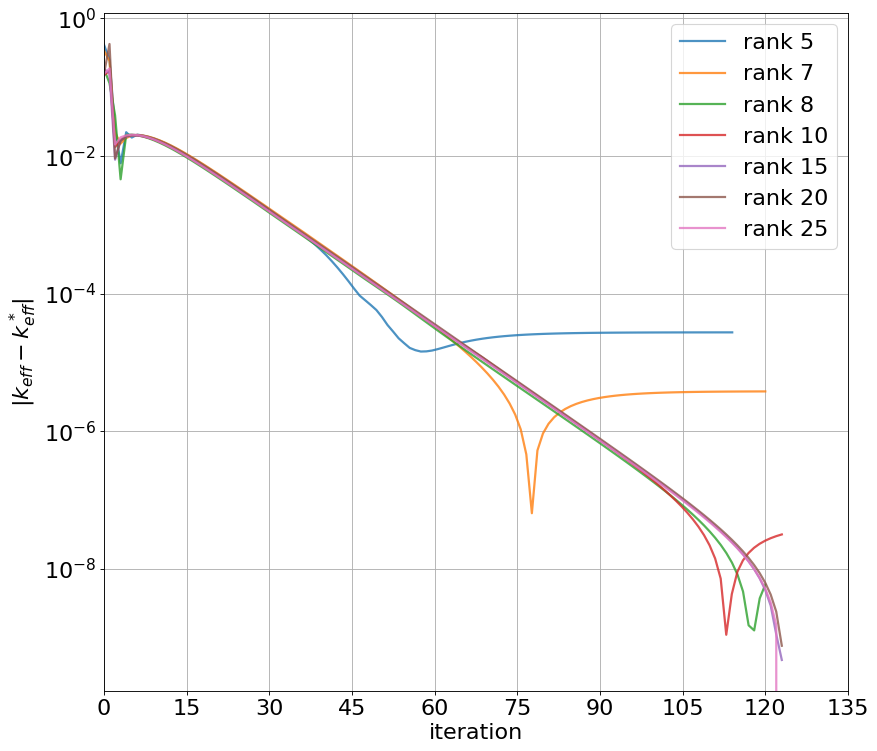}
		\caption{}
		\label{fig:historyMaterial-lwr}
		\end{subfigure}
	    \begin{subfigure}{0.49\linewidth}
	    \includegraphics[width=1.0\linewidth]{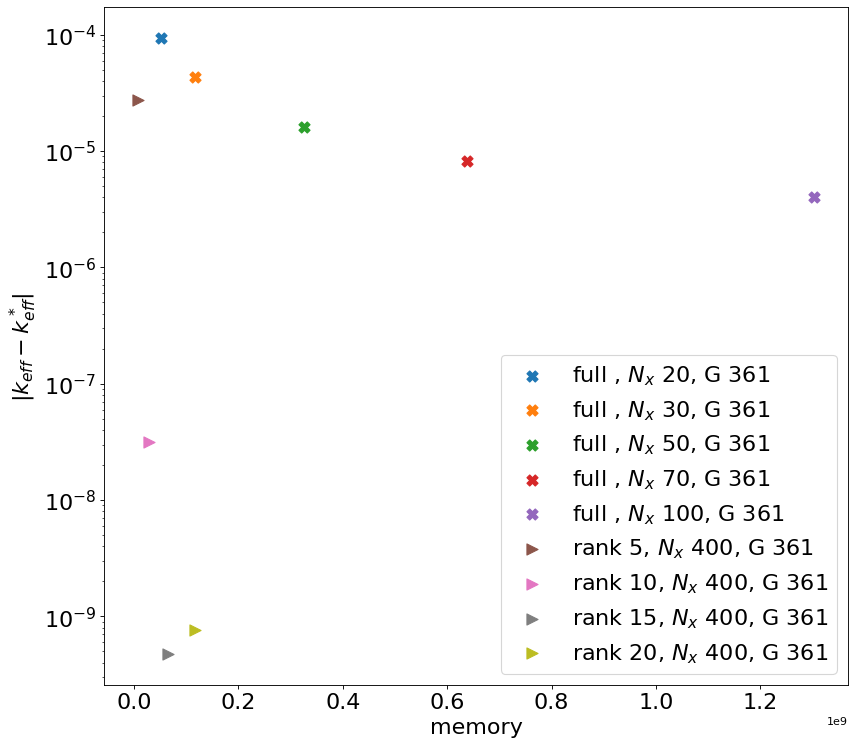}
	    \caption{}
	    \label{fig:memoryMaterial-lwr}
	    \end{subfigure}
    \caption{(a) Convergence of the effective eigenvalue for the light water reactor testcase. (b) Theoretical memory requirement of the full algorithm ($N_x^2\cdot G^2$) vs. dynamical low-rank approximation ($r^2 \cdot N_x^2+ r^2 \cdot G^2$) and corresponding error. As reference effective eigenvalue $k_\mathrm{eff}^*$, the converged solution of the rank $25$ DLRA method is used. The number of spatial cells for DLRA is $N_x = 400$. The full method runs out of memory when using more than $N_x = 100$ cells.}

\end{figure}

The spatial and energy basis functions are plotted in Figure \ref{fig:basis_light_water_reactor}. The spatial basis indicates that the leading mode peaks at the center of the problem and decays toward the boundary. The higher spatial basis functions account for the fact that different energy groups will decay at different rates as the edge of the sphere is approached. In this problem the energy basis is especially interesting because the fine group structure is able to give details on many of the resonances in the energy spectrum. Additionally, we observe in Figure \ref{fig:singularValues-lwr} that the singular values of the system decay more rapidly in the problem, reaching smaller than $10^{-13}$ by rank 25.
\begin{figure}[htp!]
\centering
    \begin{subfigure}{0.49\linewidth}
		\centering
		\includegraphics[width=\linewidth]{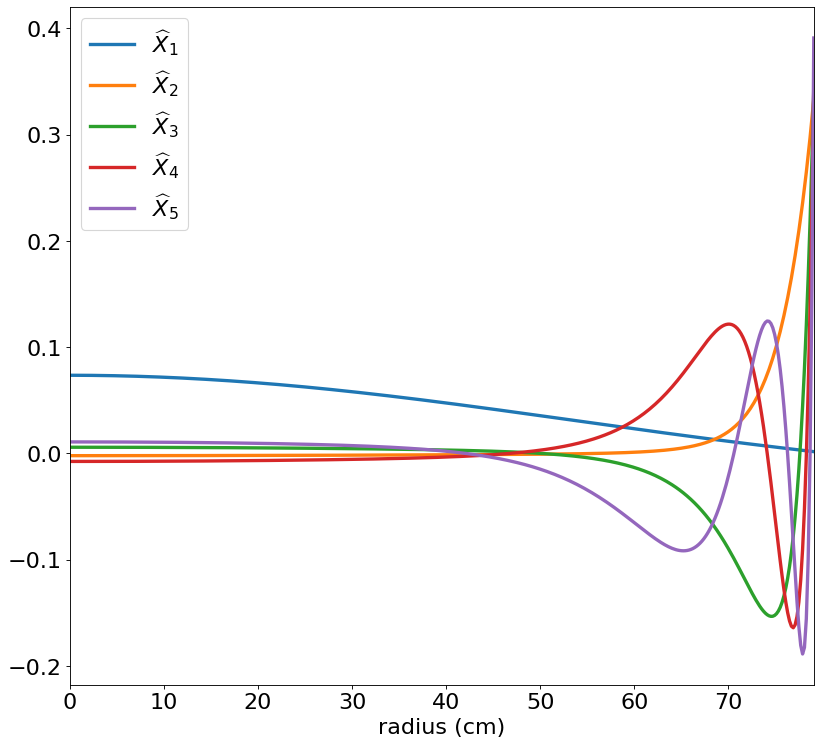}
		\caption{Basis function in radius}
		\label{fig:basisX-lwr}
	\end{subfigure}
	\begin{subfigure}{0.49\linewidth}
		\centering
		\includegraphics[width=\linewidth]{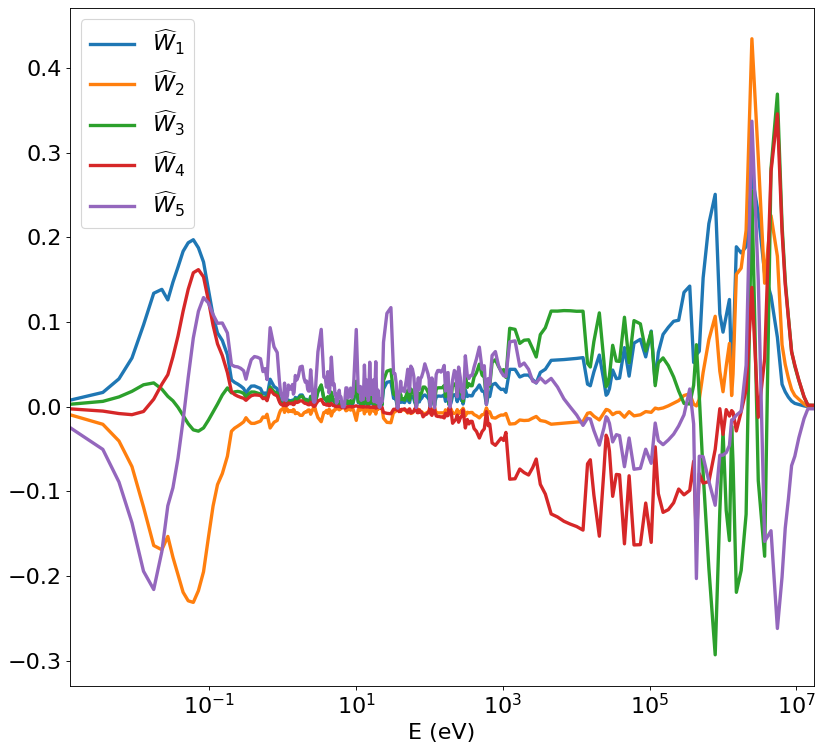}
		\caption{Basis functions in energy}
		\label{fig:basisW-lwr}
	\end{subfigure}
		\begin{subfigure}{0.49\linewidth}
		\centering
		\includegraphics[width=\linewidth]{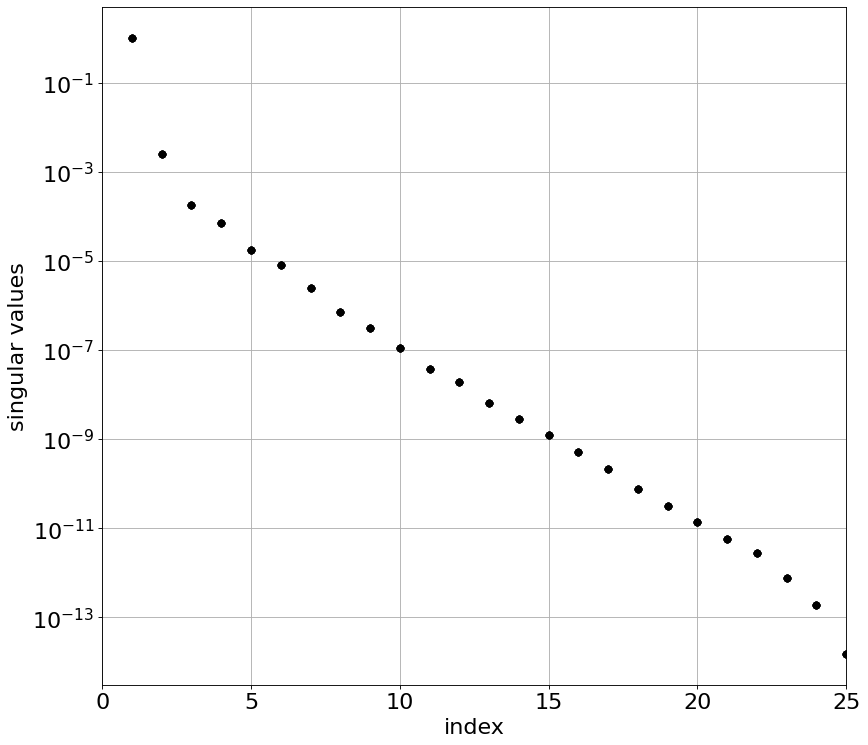}
		\caption{Singular values of the solution matrix}
		\label{fig:singularValues-lwr}
	\end{subfigure}
    \caption{First five DLRA basis functions in (a) radius and (b) energy as well as (c) singular values for the light water reactor problem $N_x = 400$, $G = 361$ and rank $r=25$.}
	\label{fig:basis_light_water_reactor}
\end{figure}

The spatial variation in the solution and the energy spectrum are plotted in Figure \ref{fig:solution-lwr}. In this figure we observe that for this single material problem, the energy spectrum shape is approximately constant with a magnitude that shifts upward as the center of the sphere is approached. We also observe that DLRA is able to capture the energy self-shielding in the epithermal range as evidences by the characteristic dips in the energy spectrum. 
\begin{figure}[htp!]
\centering
    \begin{subfigure}{0.49\linewidth}
		\centering
		\includegraphics[width=\linewidth]{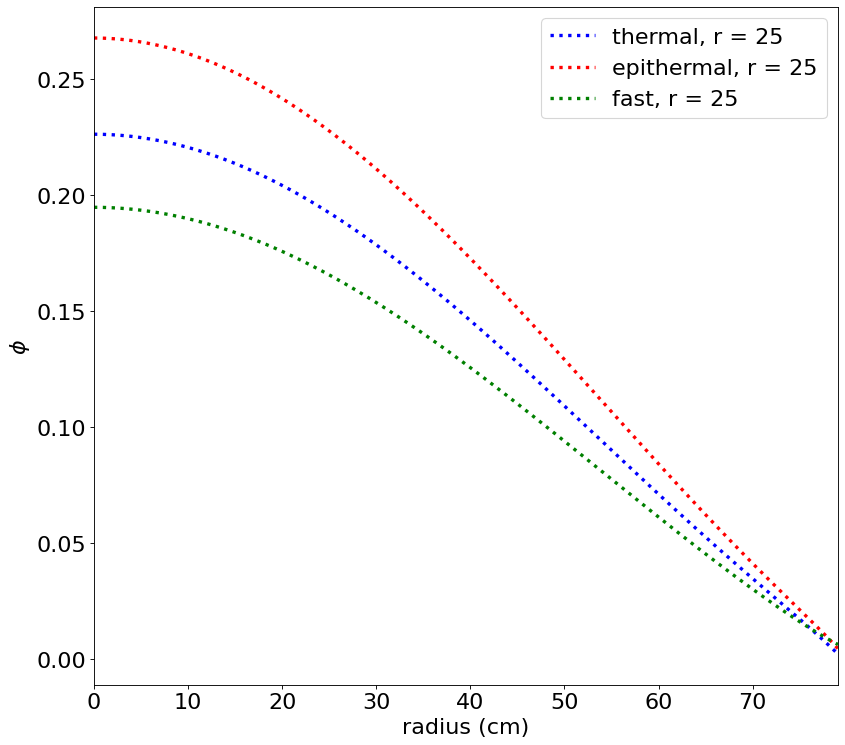}
		\caption{$N_x = 400$, $G = 361$}
		\label{fig:studyCFLPNRank15}
	\end{subfigure}
	\begin{subfigure}{0.49\linewidth}
		\centering
		\includegraphics[width=\linewidth]{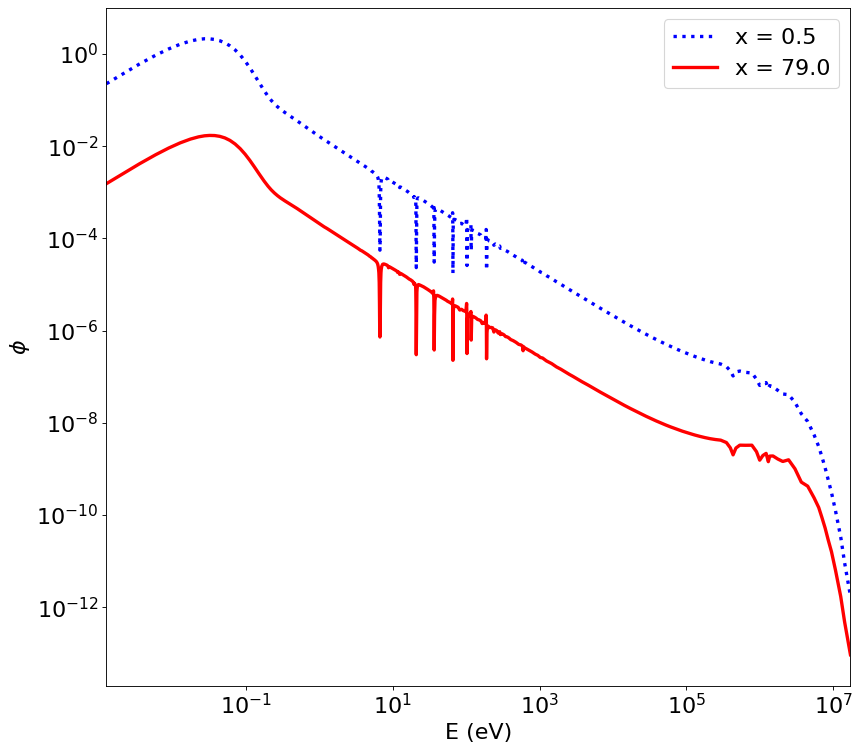}
		\caption{$\phi$ at fixed spatial positions}
		\label{fig:studyCFLPNRank10}
	\end{subfigure}
    \caption{Left: Thermal, epithermal and fast neutrons with a rank $r=25$ approximation. Right: $\phi$ at fixed spatial positions. }
	\label{fig:solution-lwr}
\end{figure}

\section{Conclusion}\label{sec:conclusion}
In this work, we presented a dynamical low-rank iteration to compute effective eigenvalues in criticality problems. The method treats the iteration index as a pseudo-time and thereby allows deriving update equations for the factorized scalar flux with the DLRA method. Consequently, the memory requirements are decreased significantly, permitting the use of fine discretizations. Our numerical experiments show that the method yields satisfactory approximations of effective eigenvalues for sufficiently high ranks. Physically relevant characteristics are captured by chosen basis functions, which encode resonance regions and relevant energy ranges.

In future work, we aim at investigating applying our techniques alongside different acceleration strategies for power iteration, e.g., coarse-mesh finite difference or Anderson acceleration. Moreover, we aim at including transport terms in the problem formulation. In this case, the solution becomes a tensor and we need employ tensor integrators as presented in \cite{CeL21} or \cite{LubichVandWalach}. 

\section*{Acknowledgments}
This work was partially funded by the Deutsche Forschungsgemeinschaft (DFG, German Research Foundation) --- Project-ID 258734477 --- SFB 1173. This work was also partially funded by the Center for Exascale Monte-Carlo Neutron Transport (CEMeNT) a PSAAP-III project funded by the Department of Energy, grant number DE-NA003967.
	
	\bibliographystyle{abbrv}
	\bibliography{main} 

\end{document}